\documentclass[12pt,twoside,final]{preprint}

\title{The complex hyperbolic form as\\ a Weil-Petersson form}
\author{Xiangsheng Wang}
\institute{School of Mathematics, Shandong University\hspace{3.5em} \email{xiangsheng@sdu.edu.cn}}
\date{\today}

\usepackage[english]{babel}
\usepackage[stretch=10,final]{microtype}
\usepackage[a4paper,left=2.5cm,right=2.5cm,top=4cm,bottom=2.5cm,marginparwidth=2cm]{geometry}
\usepackage[usenames,dvipsnames]{xcolor} \usepackage{titlesec}
\usepackage{enumitem}
\usepackage[olditem,oldenum]{paralist}
\usepackage{perpage}
\usepackage[super]{nth}
\usepackage{subdepth}
\usepackage{changepage}
\usepackage{booktabs}
\usepackage{todonotes}
\usepackage{fancyhdr}
\usepackage[title]{appendix}
\usepackage[colorlinks,linkcolor=blue,bookmarksnumbered,pagebackref,final]{hyperref}
\usepackage{graphicx}
\usepackage{subcaption}
\usepackage[]{lipsum}
\usepackage[alphabetic,initials,msc-links]{amsrefs}

\usepackage{mathtools}
\usepackage{amsthm}
\usepackage{amssymb}
\usepackage{stmaryrd} \usepackage{mleftright}
\usepackage{tikz-cd}
\usepackage{derivative}
\usepackage{accents}

\usepackage[osf,theoremfont]{newtxtext} \usepackage[upint,smallerops]{newtxmath}
\usepackage[cal=euler,scr=rsfso,frak=euler]{mathalfa}
\usepackage[T1]{fontenc}

\usepackage[normalem]{ulem}

\MakePerPage{footnote}

\DeclareSymbolFont{sfoperators}{OT1}{cmss}{m}{n}
\DeclareSymbolFontAlphabet{\mathsf}{sfoperators}
\makeatletter
\renewcommand{\operator@font}{\mathgroup\symsfoperators}
\makeatother

\numberwithin{equation}{section}
\swapnumbers                            \newtheoremstyle{plain}
{2ex plus 1ex minus .2ex}   {\medskipamount}   {\slshape}  {}       {\indent\bfseries\tlfstyle} {.}         {5pt plus 1pt minus 1pt} {}          

\newtheorem{theorem}[subsubsection]{Theorem}
\newtheorem{lemma}[subsubsection]{Lemma}
\newtheorem{corollary}[subsubsection]{Corollary}
\newtheorem{proposition}[subsubsection]{Proposition}

\newtheorem*{claim*}{Claim}
\newtheorem*{lemma*}{Lemma}

\newtheoremstyle{definition}
{2ex plus 1ex minus .2ex}   {\medskipamount}   {}  {}       {\indent\bfseries\tlfstyle} {.}         {5pt plus 1pt minus 1pt} {}          \theoremstyle{definition}

\newtheorem{remark}[subsubsection]{Remark}
\newtheorem{example}[subsubsection]{Example}
\newtheorem*{remark*}{Remark}
\newtheorem*{assumption*}{Assumption}
\newtheorem*{example*}{Example}

\makeatletter
\let\newtitle\@title
\let\newauthor\@author
\let\newdate\@date
\makeatother

\titleformat{\section}{\normalfont\large\bfseries\tlfstyle}{\thesection}{0.6em}{}
\titleformat{\subsection}[runin]{\normalfont\normalsize\bfseries\tlfstyle}{\thesubsection}{0.4em}{}
\titleformat{\subsubsection}[runin]{\normalfont\normalsize\bfseries\tlfstyle}{\indent\thesubsubsection}{0.4em}{}

\tikzcdset{
  arrow style=tikz,
  diagrams={>={Straight Barb[scale=0.8]}}
}

\makeatletter
\newcommand*{\transpose}{{\mathpalette\@transpose{}}}
\newcommand*{\@transpose}[2]{\raisebox{\depth}{$\m@th#1\intercal$}}
\makeatother

\newcommand{\Hom}[2]{\operatorname{Hom}(#1,#2)}

\newcommand{\aset}[1]{\{#1\}}

\newcommand{\simarrow}{\xrightarrow{
    \smash{\raisebox{-0.65ex}{\ensuremath{\scriptstyle\sim}}}}}

\newcommand{\cA}{\mathcal{A}}

\newcommand{\cE}{\mathcal{E}}
\newcommand{\cO}{\mathcal{O}}

\newcommand{\cM}{\mathcal{M}}

\newcommand{\cS}{\mathcal{S}}
\newcommand{\cT}{\mathcal{T}}

\newcommand{\cW}{\mathcal{W}}

\newcommand{\sC}{\mathscr{C}}

\newcommand{\sM}{\mathscr{M}}

\newcommand{\rB}{\mathrm{B}}

\newcommand{\rT}{\mathrm{T}}

\newcommand{\rL}{\mathrm{L}}
\newcommand{\rH}{\mathrm{H}}
\newcommand{\rP}{\mathrm{P}}
\newcommand{\rR}{\mathrm{R}}

\newcommand{\bZ}{\mathbb{Z}}
\newcommand{\bC}{\mathbb{C}}
\newcommand{\bN}{\mathbb{N}}
\newcommand{\bR}{\mathbb{R}}
\newcommand{\bP}{\mathbb{P}}

\newcommand{\bfu}{\mathbf{u}}

\newcommand{\bfw}{\mathbf{w}}

\newcommand{\norm}[1]{\left\Vert#1\right\Vert}
\newcommand{\kah}{K\"ahler}

\newcommand{\tchm}{Teichm\"uller}
\newcommand{\tnn}{\mathcal{T}_{0,n}}
\newcommand{\cpo}{\mathbb{CP}^1}
\newcommand{\hbC}{\hat{\mathbb{C}}}
\newcommand{\zb}{\bar{z}}
\newcommand{\wb}{\bar{w}}
\newcommand{\dzg}{\mathrm{Diff}^0_{0,n}}
\newcommand{\pdg}{\mathrm{PDiff}^+_{0,n}}
\newcommand{\pmg}{\mathrm{PMod}_{0,n}}
\newcommand{\tpi}{2\pi i}
\newcommand{\hoz}{\mathrm{H}^{1,0}}
\newcommand{\hzo}{\mathrm{H}^{0,1}}
\newcommand{\tv}{\mathrm{TV}}
\newcommand{\pw}{\mathrm{WP}}
\newcommand{\cy}{\mathrm{CY}}
\newcommand{\Xst}{X_{\mathrm{st}}}
\newcommand{\loc}{\mathrm{loc}}
\newcommand{\dzs}{\lvert \diff z \rvert^2}
\newcommand{\dxs}{\lvert \diff \xi \rvert^2}

\DeclareMathOperator{\diff}{d}

\DeclareMathOperator{\supp}{supp}

\DeclareMathOperator{\opi}{i}
\DeclareMathOperator{\opj}{j}

\DeclareMathOperator{\opc}{c}
\DeclareMathOperator{\id}{id}

\DeclareMathOperator{\opP}{P}

\DeclareMathOperator{\opT}{T}

\DeclareMathOperator{\Vol}{Vol}
\DeclareMathOperator*{\esup}{ess\,sup}

\begin{document}
\maketitle

\begin{abstract}
  For the moduli space of the punctured spheres, we find a new equality between two symplectic forms defined on it.
  Namely, by treating the elements of this moduli space as the singular Euclidean metrics on a sphere, we give an interpretation of the complex hyperbolic form, i.e.\ the K\"ahler form of the complex hyperbolic structure on the moduli space, as a kind of Weil-Petersson form.
\end{abstract}

\section{Introduction}
\label{sec:introduction}

\subsection{}
A rather interesting phenomenon about various moduli spaces is that many of them carry natural symplectic forms.
Considering this fact, if two moduli spaces are correlated in some ways, we can expect that it may be possible to refine this relation to the symplectic geometry level.
Among the instances of such a phenomenon, we are particularly inspired by the following result due to Shimura~\cite{Shimura_1959le} and Goldman~\cite{Goldman_1984aa}.
\begin{theorem}[{\cite[Proposition~2.5]{Goldman_1984aa}}]
  \label{thm:s-g}
  Let $S$ be a close surface with genus greater than $1$ and let $\cT(S)$ be the \tchm\ space of $S$.
  Denote the Weil-Petersson symplectic form on $\cT(S)$ by $\omega_{\pw}$.
  By identifying $\cT(S)$ as an open subset of $\Hom{\pi_1 (S)}{G}/G$, $G = \mathrm{SL}(2,\bR)$, $\cT(S)$ carries another symplectic form, denoted by $\omega_{\mathrm{SG}}$.
  Then $\omega_{\pw}$ and $\omega_{\mathrm{SG}}$ coincide up to a constant.
\end{theorem}

The existence of hyperbolic structures on $S$, i.e.\ metrics with the constant Gauss curvature $-1$, is crucial for the definition of both $\omega_{\mathrm{SG}}$ and $\omega_{\pw}$ in the above theorem.
To understand this, we take $\omega_{\pw}$ as an example.
Fix a complex structure $J$ on $S$, that is, $J \in \cT(S)$, which corresponds to a hyperbolic structure (metric) $g$ on $S$.
By identifying the holomorphic cotangent space of $\cT(S)$ at $J$ with the space of holomorphic quadratic differentials on $(S,J)$, there is a natural Hermitian pairing on this space,
\begin{equation}
  \label{eq:hp}
  H_{\pw}(\varphi,\psi) = \int_S g^{-1}(\varphi,\bar{\psi}),
\end{equation}
where $\varphi,\psi$ are holomorphic quadratic differentials.
Then, the symplectic form $\omega_{\pw}$ for the tangent space of $\cT(S)$ at $J$ is induced from the pairing (\ref{eq:hp}).

Inspired by (\ref{eq:hp}), a natural question is whether one can generalize Theorem~\ref{thm:s-g} by replacing the hyperbolic structure with some weaker structures.
Mondello~\cite{Mondello_2010po} succeeds in obtaining such a generalization by considering the hyperbolic metrics admitting cone singularities.
\begin{theorem}[{\cite[Theorem~1.2]{Mondello_2010po}}]
  \label{thm:mon}
  Let $(S,x)$ be a closed surface with the marked points $x = (x_1,\cdots, x_n)$ and let $\cT(S,\theta)$ be the \tchm\ space of the hyperbolic metrics on $S$ with angles $\theta = (\theta_1,\cdots,\theta_n)$ at $x$.
  Denote the Weil-Petersson symplectic form on $\cT(S,\theta)$ by $\omega_{\pw}$.
  By using the local injection from $\cT(S,\theta)$ to $\Hom{\pi_1 (S - x)}{G}/G$, $G = \mathrm{SL}(2,\bR)$, $\cT(S,\theta)$ carries another symplectic form, denoted by $\omega_{\mathrm{SG}}$.
  Then for small enough $\theta$, $\omega_{\pw}$ and $\omega_{\mathrm{SG}}$ coincides up to a constant.
\end{theorem}

Ever since Troyanov's work~\cite{Troyanov_1986le}, people have learned a lot about the metrics with cone singularities.
Especially, many results about the smooth metrics can be generalized to this special singular case.
The relation between the above two theorems provides an example of this general philosophy.

Admittedly, compared with smooth metrics, metrics with cone singularities are more complicated analytically.
However, in some sense, the metrics with cone singularities are more flexible than their smooth counterpart, which is especially useful when we are considering the metrics with curvature constraints, e.g.\ with constant scalar curvature.
For example, in Theorem~\ref{thm:mon}, this extra flexibility is used to break the genus constraints imposed by the Gauss-Bonnet formula.

This paper tries to exploit the flexibility of metrics with cone singularities in another way.
Recall that most surfaces do not carry a flat metric.
But, if we admit cone singularities, singular flat metrics always exist when the cone angles and the genus satisfy a Gauss-Bonnet type equality.
Motivated by this fact, this paper develops a ``flat'' variant of Theorem~\ref{thm:mon}.

\subsection{}
To find a ``flat'' variant of Theorem~\ref{thm:mon}, we need to find proper substitutes for $\omega_{\pw}$ and $\omega_{\mathrm{SG}}$.
Let $\cT_{g,n}$ be the \tchm\ space of genus $g$ closed surfaces with $n$ marked points.
It is known that $\cT_{g,n}$ is isomorphic to the moduli space of the flat metrics with cone singularities (also called singular Euclidean metrics) on such a surface, which we will recall in Section~\ref{sec:moduli} for $g= 0$ case.
Therefore, we can define the Weil-Petersson metric (symplectic form) on $\cT_{g,n}$ using the singular Euclidean metric in the same fashion as (\ref{eq:hp}) in the hyperbolic case.

In contrast, in the current situation, a proper substitute for $\omega_{\mathrm{SG}}$ is not so obvious.
Similar to the hyperbolic case, there is a map from $\cT_{0,n}$ to $\Hom{\pi_1 (S - x)}{\mathrm{SE} (2,\bR)}/\mathrm{Sim} (2)$, where $\mathrm{SE} (2,\bR)$ is the holonomy group for the flat metrics and $\mathrm{Sim} (2) \coloneqq \bR_+ \ltimes {\mathrm{SE} (2,\bR)}$, c.f.~\cite{Troyanov_2007th}.
However, this map does not provide a symplectic form on $\tnn$ via pull-back as in~\cite{Goldman_1984aa,Mondello_2010po}.
The problem is that $\mathrm{SE} (2,\bR)$ is not a semisimple group, that is, this group does not have a non-degenerate Killing form.
As a result, unlike $\omega_{\mathrm{SG}}$, we cannot use the usual method of defining a symplectic form on $\Hom{\pi_1 (S - x)}{\mathrm{SE} (2,\bR)}/{\mathrm{Sim} (2)}$.

Nevertheless, for the genus zero case, i.e.\ $S$ is a sphere, a special symplectic form, which we call the complex hyperbolic form, does play the role of $\omega_{\mathrm{SG}}$.
In fact, a sphere equipped with a flat metric with cone singularities is equivalent to a more familiar concept, a polyhedron.
From this viewpoint, the moduli space of flat metrics with cone singularities has undergone comprehensive study, c.f.~\cite{Thurston_1998sh,Veech_1993fl}.
Roughly speaking, the logarithm of the area of polyhedra provides a \kah\ potential on the moduli space and the complex hyperbolic form is the \kah\ form of such a \kah\ potential.
We will define this form precisely in Section~\ref{sec:moduli}.

The main result that we are going to show, that is, a ``flat'' variant of Theorem~\ref{thm:mon}, is an identification between two symplectic forms on $\cT_{0,n}$.

\begin{theorem}
  \label{thm:main}
  Let $x = \aset{x_0, \cdots, x_{n-1}}$ be the marked points on the two-dimensional sphere $S^2$ and let $\alpha = \aset{\alpha_0,\cdots,\alpha_{n-1}}$ be a set of real numbers such that $0 < \alpha_i < 1$, $i = 0, \cdots,n-1$, and $\sum_{i=0}^{n-1}\alpha_i = 2$.
Denote $\tnn$ to be the \tchm\ space of the flat metrics with cone angles $\alpha$ at $x$.
  Then, on $\tnn$, the complex hyperbolic form coincides with the Weil-Petersson form.
\end{theorem}

Because both the complex hyperbolic form and the Weil-Petersson form are \kah\ forms for certain metrics on $\tnn$, the above theorem is equivalent to saying that these two metrics coincide.
Moreover, since the complex hyperbolic form is closed by definition, we have the following corollary of Theorem~\ref{thm:main}.
\begin{corollary}
  \label{cor:main}
  The Weil-Petersson form on $\tnn$ defined using the flat metrics with cone angles is closed.
\end{corollary}

Besides the relation between Theorem~\ref{thm:main} and Goldman and Mondello's results, Theorem~\ref{thm:main} can also be viewed as a variant of a seminal result due to Tian~\cite[Theorem~2]{Tian_1987aa}.
In fact, such a similarity inspires the proof of Theorem~\ref{thm:main}.
We will explain it in more detail in Section~\ref{sec:proof-of-main}.

\subsection{An outline of the paper.}
This paper is organized as follows.
In Section~\ref{sec:moduli}, we review the definition of the flat metric with cone singularities and the moduli space of such metrics.
Especially, we recall the definition of the complex hyperbolic form and the Weil-Petersson form.
In Section~\ref{sec:some-ana-res}, we show some analytic results about an integral operator needed in the proof of Theorem~\ref{thm:main}.
In Section~\ref{sec:proof-of-main}, we explain the relation between Theorem~\ref{thm:main} and Tian's result and finish the proof of it.

\subsection*{Acknowledgement.}
The author would like to thank Prof.\ Gang Tian for helpful suggestions and encouragement during writing this paper and the anonymous referee for reading the paper carefully and the very inspiring comments.
The author is partially supported by NSFC Grant no. 12101361, the project of Young Scholars of SDU, and the Fundamental Research Funds of SDU, Grant no. 2020GN063.

\section{Moduli space of singular Euclidean metrics on a sphere}
\label{sec:moduli}

\subsection{}
In the section, we set up the notations used in this paper and review the definition of a singular Euclidean metric on a sphere.
Especially, we discuss the moduli space of such metrics, which can be identified with the usual \tchm\ space $\tnn$.
Next, following the approach of~\cite{Deligne_1986mo}, we will define the complex hyperbolic form on $\tnn$ by using such an identification.
Meanwhile, we will also introduce a variant of the usual Weil-Petersson form on $\tnn$ by replacing the hyperbolic metric with the singular Euclidean metric.

\subsection{}
In the whole paper, we will identify the two-dimensional sphere $S^2$ with $ \bC \cup \aset{\infty}$.
Let $n > 3$ be an integer.
On $S^2$, we choose $z_0 = 0$, $z_1 = 1$, \ldots, $z_{n-2} = n-2$ and $z_{n-1} = \infty$ as $n$ marked points.
We denote the punctured sphere associated with the chosen marked points by $X \coloneqq S^2 - \aset{z_0,z_1,\cdots,z_{n-1}}$.

For various places, we need to specify a complex structure on $S^2$.
Especially, for $S^2$ carrying the standard complex structure on $\bC \cup \aset{\infty} = \cpo$, we also use the notation $\hbC$.
Therefore, as a subset of $\hbC$, $X$ also carries a complex structure.
$X$ with such a complex structure, if necessary, will be denoted by $\Xst$.\footnote{Certainly, on $X$, complex structures other than $\Xst$ exist.
For example, let $Y$ be a surface defined by $\cpo$ removing punctures different from $z_i$ and $F$ be a diffeomorphism between $X$ and $Y$.
By pulling back the complex structure on $Y$ by $F$, $X$ has another complex structure.}

\subsection{Flat cone spheres.}
Let $\alpha = \aset{\alpha_0,\alpha_1,\cdots,\alpha_{n-1}}$ be a set of real numbers.
For a Riemannian metric $g$ on $X \subseteq S^2$, we call $g$ \emph{a flat metric on $S^2$ with the cone angle $\alpha_i$ at $z_i$, $i = 0,1,\cdots,n-1$},\footnote{Some authors may use different convention to define cone angles in literature.
  For example, the cone angle $\theta_i$ in~\cite{Mondello_2010po} equals to $2\pi(1- \alpha_i)$.}
if the following conditions hold:
\begin{enumerate}[label=(\arabic*)]
\item There is a complex structure on $S^2$ associated with $g$, which provides a local holomorphic coordinate system near each point of $S^2$.
\item Near any point of $S^2$ (Not only $X$!), $g = e^{\gamma} \diff w\diff \bar{w}$ holds locally, where $w$ is a local holomorphic coordinate and $\gamma$ is a locally defined smooth function away from $z_i$, $ i= 0,1,\cdots, n-1$.
\item The Gauss curvature of $g$ is equal to $0$, i.e.\ $\pdv{\gamma}{w,\wb} =0$ in the local holomorphic coordinate.
\item For $ i= 0,1,\cdots, n-1$, $g$ has the asymptotics near the singularity $z_i$
  \begin{equation}
    \label{eq:asymp}
    \gamma =
-2\alpha_i \log{|w(z)- w(z_i)|} + \cO(1),\; z \rightarrow z_i,
\end{equation}
  using the holomorphic coordinate around $z_i$.
\end{enumerate}
We call that the complex structure in the above definition is \emph{compatible} with $g$.

As we have used in the Introduction, in this paper, a \emph{singular Euclidean metrics} on $S^2$ is the synonym of a flat metric on $S^2$ with the cone angle $\alpha_i$ at $z_i$, $i = 0,1,\cdots,n-1$.

The above definition for the singular Euclidean metrics also works for a general surface.
For the sphere case that we are mainly concerned with, we also call ${(S^2,g)}$ \emph{a flat cone sphere} if $g$ is a singular Euclidean metric.
A good reference about the singular Euclidean metrics is~\cite{Troyanov_2007th}.

The Gauss-Bonnet theorem gives a necessary condition on cone angles for the existence of singular Euclidean metrics on $S^2$,
\begin{equation}
  \label{eq:gb-cond}
  \sum^{n-1}_{i=0} \alpha_i  = 2.
\end{equation}
It turns out that the above condition is also sufficient for the existence problem, which is guaranteed by the following theorem of Troyanov, also see~\cite[p. 90, Th{\'e}or{\`e}me]{Troyanov_1986le}.

\begin{theorem}[{\cite[Theorem~2.22]{Troyanov_2007th}}]
  \label{thm:ty}
  Let $S$ be a closed connected oriented surface. Fix $n$ distinct points $p_1, p_2, \cdots, p_n \in S$ and $n$ real numbers $\alpha_1, \alpha_2, \cdots,\allowbreak \alpha_n \in (-\infty,1)$.

  Fix a complex structure on $S$.
  There exists a flat metric $g$ on $S$ with a conical singularity of cone angle $\alpha_j$ at $p_j$ $(j = 1, \cdots, n)$ and the chosen complex structure is compatible with $g$ if and only if the Gauss–Bonnet condition $\chi (S) = \Sigma_{j=1}^n\alpha_j$ holds.
  This metric is unique up to homothety.
\end{theorem}

\begin{example}
  If the chosen complex structure on $S^2$ is the standard one, i.e.\ $\hbC$, for any cone angles $\aset{\alpha_i}$ satisfying (\ref{eq:gb-cond}), we can write (the \kah\ form of) the flat metric $g$ asserted by the above theorem explicitly,
  \begin{equation}
    \label{eq:sp-metric}
    g = \prod^{n-2}_{i = 0} |z - z_i|^{-2\alpha_i} \dzs,
  \end{equation}
  where and in the following, we use $\dzs$ as a shorthand for $\frac{\diff z \wedge \diff \zb}{-\tpi}$.

  One can check that $g$ has the desired singularity at $z_{n-1} = \infty$.
  By Theorem~\ref{thm:ty}, we know that such a metric is unique up to homothety.
\end{example}

\noindent\textbf{Assumption.}
In the rest of the paper, we fix $\alpha = \aset{\alpha_0,\alpha_1,\cdots,\alpha_{n-1}}$ to be a set of real numbers satisfying the Gauss-Bonnet condition (\ref{eq:gb-cond}) and
\begin{equation}
  \label{eq:alpha-range}
  0 < \alpha_i < 1,\quad i = 0,1,\cdots,n-1.
\end{equation}

\begin{remark}
  Note that by Theorem~\ref{thm:ty}, the flat metric exists for a wider range of $\alpha$ than what we have assumed in (\ref{eq:alpha-range}).
  However, similar to Theorem~\ref{thm:mon}, to prove Theorem~\ref{thm:main}, we have to restrict $\alpha$ to a narrower range.
  In fact, even for the definition of the two symplectic forms appearing in Theorem~\ref{thm:main}, (\ref{eq:alpha-range}) is also necessary.
\end{remark}

\subsection{Moduli spaces.}
\label{sub:moduli-space}
We are interested in the moduli space consisting of all flat cone spheres.
There are several definitions of this moduli space in literature.
Here, we follow the approach of~\cite{Troyanov_2007th,Veech_1993fl}, which is different from, though closely related to, the definitions given in~\cite{Schwartz_2015no,Thurston_1998sh}.

Let $\cE(\alpha)$ be the set of all flat metrics on $S^2$ with cone angles $\alpha = \aset{\alpha_0,\cdots,\allowbreak\alpha_{n-1}}$ at punctures $z_0,\cdots,z_{n-1}$.
We denote $\dzg$ to be the group of diffeomorphisms that are isotopic to the identity through isotopies fixing the punctures.
Via the pull-back metrics, $\dzg$ can act on $\cE(\alpha)$.
Therefore, we can define a group action of $\bR_+ \times \dzg$ as follows:
\begin{equation*}
  \begin{aligned}
    (\bR_+ \times \dzg) &\times \cE(\alpha) && \rightarrow \cE(\alpha)\\
    (\lambda, h) &\times m && \mapsto \lambda h^*m
  \end{aligned}.
\end{equation*}

We define \emph{the (\tchm) moduli space of flat cones spheres} to be the quotient of such an action
\begin{equation*}
  \cS(\alpha)\coloneqq \cE(\alpha)/(\bR_+ \times \dzg).
\end{equation*}
Recall that one way to define the \tchm\ space of the punctured spheres is
\begin{equation*}
  \cT_{0,n} \coloneqq \aset{\text{complex structures on }S^2}/\dzg.
\end{equation*}
One corollary of Theorem~\ref{thm:ty} is that $\cS(\alpha)$ is canonically isomorphic to $\cT_{0,n}$.
This paper will use $\cS(\alpha)$ and $\tnn$ interchangeably.

\begin{remark}
  \label{rk:wn}
  As usual, let $\pdg$ be the group of orientation-preserving diffeomorphisms fixing the punctures and $\pmg \coloneqq \pdg/ \dzg$ be the pure mapping class group.
  As in~\cite[\S~5.1]{Troyanov_2007th}, we define the moduli space of punctured spheres to be\footnote{For the definition of the moduli space, the definition given here follows the conventions used in the algebraic geometry as~\cite{Arbarello_2011ge}.
    However, in the literature of \tchm\ theory, the (Riemann) moduli space is usually defined using the full mapping class group and $\cM_{0,n}$ defined in this way is sometimes called the Torelli space, e.g.~\cite{Nag_1988co}.}
  \begin{equation*}
    \cM_{0,n} \coloneqq \cT_{0,n} / \pmg.
  \end{equation*}
  Moreover, we denote $\cW_n$ to be the configuration space of punctures on $\hbC^n$,
  \begin{equation*}
    \cW_n \coloneqq \aset{\bfw = (w_0, w_1,\cdots, w_{n-1})\in \hbC^{n}| w_0 = 0,\, w_1=1,\, w_{n-1}= \infty \text{ and }w_i \neq w_k\text{ for }i\neq k}.
  \end{equation*}
By this definition, $\cW_n$ is an open subset of $\bC^{n-3}$.
  It is known that $\cM_{0,n}$ is isomorphic to $\cW_n$, see~\cite[\S~6.1]{Troyanov_2007th} or~\cite{Patterson_1973so,Nag_1988co}.

  In~\cite{Schwartz_2015no} (and in~\cite{Thurston_1998sh} implicitly), the author uses a concept called labeled moduli space of flat cone spheres $\sM(\alpha)$.
  In our notation, $\sM(\alpha)$ is just $\cS(\alpha)/\pmg$.
  Therefore, $\sM(\alpha)$ is canonically isomorphic to $\cM_{0,n}$ or $\cW_n$.
\end{remark}

By identifying $\cS(\alpha)$ with $\cT_{0,n}$, we can review the construction of symplectic forms on $\cT_{0,n}$ appearing in Theorem~\ref{thm:main}.
For this purpose, we will follow the approach of~\cite{Deligne_1986mo} mainly.
Compared with the method in~\cite{Thurston_1998sh}, this approach does not relate to the singular Euclidean metrics so directly, whereas it is in the same vein as the construction of the period map, c.f.~\cite{Voisin_2002aa}.
Such a resemblance is crucial for our proof of Theorem~\ref{thm:main}.
As a preparation, we review some facts about the cohomology groups of a local system on $X$ associated with the metric (\ref{eq:sp-metric}).

\subsection{Cohomology groups of a local system on $X$.}
\label{sub:hm}
Let $\sigma_k = \exp (\tpi \alpha_k)$, $k= 0,1,\cdots,n-1$.
The Gauss-Bonnet condition (\ref{eq:gb-cond}) implies $\prod_k \sigma_k = 1$ and (\ref{eq:alpha-range}) implies $\sigma_k \neq 1$.
As a result, we can find a rank one local system (equivalently a complex flat line bundle) $F$ on $X$ such that the monodromy of $F$ around $z_k$ is the multiplication by $\sigma_k$.
Note that if $F'$ is another rank one local system satisfying this property, $F'$ is isomorphic to $F$ although the isomorphism between $F$ and $F'$ is not canonical.

The local system $F$ has a close relation with the metric $g$ given in (\ref{eq:sp-metric}).
Following~\cite[\S~2]{Deligne_1986mo}, we review some facts about $\rH^*(X,F)$, the cohomology groups on $X$ with coefficients in $F$, which are needed for the definition of the complex hyperbolic form.

Firstly, as a result of the combinatorial description of $\rH^*(X,F)$, the Euler characteristics $\chi(X,F)$ is equal to $2-n$.
Moreover, since $\sigma_k \ne 1$, $\rH^*(X,F)$ vanishes except for the first cohomology group, c.f.~\cite[Proposition~(2.3.1)]{Deligne_1986mo}.
Thus, we know that
\begin{equation*}
  \dim \rH^1(X,F) = n - 2.
\end{equation*}
Meanwhile, we can also view $\rH^1(X,F)$ as a smooth de Rham cohomology group, i.e.\ treating elements of $\rH^1(X,F)$ as the de Rham cohomology classes of $1$-forms valued in $F$.
Let $\rH^i_c(X,F)$ be the $F$-valued de Rham cohomology group with compact supports, which is also nonvanishing for $i=1$ alone.
Since $\sigma_k \ne 1$, we can show the following isomorphism, c.f.~\cite[Proposition (2.6.1)]{Deligne_1986mo},
\begin{equation}
  \label{eq:coh-eq}
  \opj:\rH^*_c(X,F) \simarrow \rH^*(X,F).
\end{equation}

Let $\cA^*(X,F)$ be the chain complex of the $F$-valued $\sC^{\infty}$ forms on $X$.
Since the absolute value of $\sigma_k$ is $1$, there is a horizontal, i.e.\ invariant under the monodromy, Hermitian metric $\langle -,-\rangle$ on $F$.
Such a Hermitian structure on $F$ induces a product on $\cA^1(X,F)$ valued in 2-forms, denoted by $\langle -, - \rangle$ again,
\begin{equation*}
  \langle \omega_1\otimes f_1 , \omega_2 \otimes f_2 \rangle \coloneqq \omega_1 \wedge \bar{\omega}_2 \langle f_1, f_2 \rangle.
\end{equation*}
Then, we can define a non-degenerate Hermitian form on $\rH^1_c(X,F)$,
\begin{equation*}
  (u,v) \coloneqq \frac{-1}{\tpi}\int_X \langle u,{v}\rangle,
\end{equation*}
where $u,v\in \rH^1_c(X,F)$.
The above Hermitian form also induces a Hermitian form on $\rH^1(X,F)$ owing to (\ref{eq:coh-eq}).

Fixing a complex structure on $S^2$, following~\cite{Deligne_1986mo}, we write $j^m_*\Omega^*(F)$ for the sheaf over $S^2$ consisting of certain meromorphic $F$-valued forms.
More precisely, an $F$-valued $0$-form (resp.\ $1$-form) $u$ belongs to $\Gamma(S^2,j^m_* \Omega^0(F))$ (resp.\ $\Gamma(S^2,j^m_* \Omega^1(F))$) if and only if
\begin{inparaenum}[(i)]
\item $u$ is holomorphic on $X$;
\item near $z_k$, $k = 1, \cdots,n$, $u$ can be written as $w^{-\alpha_k}f\cdot s$ (resp.\ $w^{-\alpha_k}f\cdot s\diff w$), where $w$ is a local holomorphic coordinate near $z_k$ satisfying $w(z_k) = 0$, $s$ is a nonzero multivalued horizontal section\footnote{Readers can see~\cite[(2.8)]{Deligne_1986mo} for a precise definition of the multivalued horizontal section. Roughly speaking, they can be viewed as the ``global horizontal sections'' on a flat line bundle with the nontrivial monodromy.} of $F$ near $z_k$ and $f$ is a meromorphic function near $z_k$.
\end{inparaenum}
If $u\in \Gamma(S^2,j^m_* \Omega^*(F))$, we define its valuation at $z_k$ to be
\begin{equation}
  \label{eq:def-val}
  v_{z_k}(u) = v_{z_k}(f) - \alpha_k.
\end{equation}

To calculate $\rH^*(X,F)$, we can use meromorphic forms defined in the previous paragraph rather than the usual $\sC^{\infty}$ forms. More precisely, by~\cite[(2.10.2)]{Deligne_1986mo}, we have
\begin{equation*}
  \rH^*(X,F) = \rH^* \big(\Gamma(S^2,j^m_* \Omega^*(F))\big).
\end{equation*}
This isomorphism is particularly useful for calculating the Hermitian form on $\rH^1(X,F)$.
To see this, let $\hoz(X,F)$ be the subspace of $\Gamma(S^2,j^m_* \Omega^1(F))$ consisting of sections whose valuation at $z_k$, $k=1,\cdots,n$, are greater than $-1$ and let $\hzo(X,F) \subseteq \cA^1(X,F)$ be the complex conjugate of $\hoz(X,\bar{F})$.
Note that any element of $\hoz(X,F)$ and $\hzo(X,F)$ is closed, thus defining a de Rham cohomology class.
In fact, we have the following isomorphism similar to the Hodge decomposition, c.f.~\cite[Proposition (2.20)]{Deligne_1986mo},
\begin{equation*}
  \hoz(X,F) \oplus \hzo(X,F) \simarrow \rH^1(X,F) = \rH^1_c(X,F).
\end{equation*}
By direct calculation, we know the following two facts about the above decomposition,
\begin{itemize}
\item $\dim \hoz(X,F) = 1$ and $\dim \hzo(X,F) = n-3$;
\item the restriction of the Hermitian form of $\rH^1(X,F)$ on $\hoz(X,F)$ (resp.\ $\hzo(X,F)$) is positive (negative) definite.
\end{itemize}
At last, we would like to comment that although some constructions in this subsection may work for a wider range of cone angles, we still need (\ref{eq:alpha-range}) to ensure that the cohomology groups satisfy the above two properties, c.f.~\cite[(3.10)]{Deligne_1986mo}.

\subsection{The complex hyperbolic form.}
\label{sub:def-tv}
Recall that we have defined the moduli space $\cW_n$ in Remark~\ref{rk:wn}, a quotient space of $\tnn$.
To define the complex hyperbolic form, we also need a geometric construction associated with $\cW_n$ given in~\cite[\S~3]{Deligne_1986mo}.

Over $\cW_n$, there is a universal family of punctured spheres,
\begin{equation*}
  \rP_{\cW_n} \coloneqq \aset{(p,\bfw)\in \hbC \times \cW_n| p \notin \bfw}.
\end{equation*}
We write $\pi$ for the projection from $\rP_{\cW_n}$ to $\cW_n$.
Let $U$ be a contractible small neighborhood of $\bfw \in \cW_n$.
Then one can find a (not unique) rank one local system $L$ on $\pi^{-1}(U)$ such that for any $\bfu \in U$, $L_{\bfu} \coloneqq L|_{\pi^{-1}(\bfu)}$ is a rank one local system for $\rP_{\bfu} \coloneqq \pi^{-1}(\bfu)$ with the monodromy being $\sigma_i$ at $u_i$.
As usual, $\rR^1\pi_* L$ is a local system on $U$ whose fiber at $\bfu$ is $\rH^1(\rP_{\bfu},L_{\bfu})$.
We can take the projectivization of $\rR^1\pi_* L$ fiberwisely and obtain the space $\bP \rR^1\pi_* L$.
Then $\bP \rR^1\pi_* L$ is independent of the choice of $L$.
As a result, we can glue $\bP \rR^1\pi_* L$ on different $U$ to obtain a global defined flat projective space bundle over $\cW_n$, which is denoted by $\rB(\sigma)$.

A key feature of bundle $\rB(\sigma) \rightarrow \cW_n$ is that it admits a canonically defined holomorphic section $b_{\alpha}$.
On $U$, $b_{\alpha}$ is the projective equivalence class of (the cohomology class of) the following form, c.f.~\cite[Lemma (3.5)]{Deligne_1986mo},
\begin{equation}
  \label{eq:balpha}
  \eta(\bfu) \coloneqq \prod^{n-2}_{k=0} (z - u_k)^{-\alpha_k}\cdot e \diff z,
\end{equation}
where $z$ is the standard holomorphic coordinate on $\hbC$, $e$ is a multivalued horizontal section of $L$, and $\bfu = (u_2,\cdots,u_{n-2})$ is the natural coordinates on $U$ as a subset of $\bC^{n-3}$.

Choose $o = (z_0,z_1,\cdots, z_{n-1})\in \cW_n$ as our base point.
Let $\rB(\sigma)_o$ be the fiber over $o$.
Because $\cT_{0,n}$ is the universal covering of $\cW_n$, the pull-back of $\rB(\sigma)$ onto $\cT_{0,n}$ is isomorphic to $\cT_{0,n} \times \rB(\sigma)_o$.
Taking $\tilde{b}_{\alpha}$ to be the pull-back of $b_{\alpha}$, by~\cite[Proposition (3.9)]{Deligne_1986mo}, we know that $\tilde{b}_{\alpha}: \cT_{0,n} \rightarrow \rB(\sigma)_o$ is a local isomorphism.
As in \S~\ref{sub:hm}, we equip $L$ with a horizontal Hermitian metric, which induces a Hermitian form on $\rH^1(\rP_o, L_o)$ of signature $(1,n-3)$.
Let $\rB^+(\sigma)_o$ be the subset of $\rB(\sigma)_o$ consisting of the positive elements in $\rH^1(\rP_o, L_o)$ with respect to the Hermitian form.
Then, $\rB^+(\sigma)_o$ has a complex hyperbolic structure, c.f.~\cite[Ch. 3]{Goldman_1999co}, and a companion \kah\ form $\omega$.
By the definition of $b_{\alpha}$ and the connectivity of $\tnn$, the range of $\tilde{b}^*_{\alpha}$ is contained in $\rB^+(\sigma)_o$.
Then, we define the \emph{complex hyperbolic form} on $\cT_{0,n}$ to be
\begin{equation*}
  \omega_{\tv} \coloneqq \tilde{b}^*_{\alpha}(\omega).
\end{equation*}
Since $\rB(\sigma)$ is a flat projective space bundle over $\cW_n$, one can check that the above definition is independent of the choice of the base point $o$.

For our later usage, we write $\omega_{\tv}$ in a more concrete form.
As before, we take $\bfu = (u_2,\cdots,u_{n-2})$ to be the global coordinates on $\cW_n$ induced from $\bC^{n-3}$.
Via the covering map, $\aset{u_i}$ also induces local coordinates near any point of of $\tnn$, denoted by the same symbols.
With such local coordinates near $\Xst\in \tnn$, using the local expression of $b_{\alpha}$ and the definition of the \kah\ form on the complex hyperbolic space,~\cite[(3.1.3)]{Goldman_1999co}, $\omega_{\tv}$ has the following local expression,
\begin{multline}
  \label{eq:tv-local}
  -i\omega_{\tv} = \partial_{\bfu}\bar{\partial}_{\bfu}(\log{\int_{X_{\bfu}} \frac{\eta(\bfu) \wedge \overline{\eta(\bfu)}}{-\tpi}}) \\
  = \sum_{2 \le j,k\le n-2}\pdv*{}{u_j,\bar{u}_k}{\Big(\log{(\int_{X_{\bfu}} \prod^{n-2}_{l=0} |z- u_l|^{-2\alpha_l} \dzs)}\Big)} \diff u_j \wedge \diff \bar{u}_k,
\end{multline}
where $X_{\bfu} = \bC - \aset{0,1,u_2,\cdots,u_{n-2}}$.

\begin{remark}
  As we have sketched in the Introduction, each point in $\cS(\alpha) \simeq \tnn$ can be viewed as a polyhedron.
  Using the local coordinate $\bfu$ as in (\ref{eq:tv-local}), we can see that $\int_{X_{\bfu}} \eta(\bfu) \wedge \overline{\eta(\bfu)}$ (up to a constant) is just the area of the polyhedron corresponding to $\bfu$.
  Thus, from (\ref{eq:tv-local}), we know that the logarithm of the area of polyhedra is the \kah\ potential of $\omega_{\tv}$.

  We would also like to add a historic remark here.
  As we have seen, Deligne and Mostow's method leads to a natural definition of the complex hyperbolic form.
  However, it seems that Thurston, in the preprint of~\cite{Thurston_1998sh}, first defined the complex hyperbolic metric on $\tnn$ explicitly in literature.
  Later on, Veech~\cite{Veech_1993fl} generalized Thurston's construction to a broader class.
\end{remark}

\subsection{A variant of the Weil-Petersson form.}
\label{sub:weil-petersson-form}
Now, by imitating the method of defining the usual Weil-Petersson form using the hyperbolic metrics, we construct another natural symplectic form on $\cT_{0,n}$.
We begin with recalling some facts about the \tchm\ space $\cT_{0.n}$.
As before, the base point we choose for $\cT_{0,n}$ is (the equivalence class of) $\Xst$.
By \tchm's lemma, c.f.~\cite[Proposition~6.6.2]{Hubbard_2006te} or~\cite[Lemma 7.6]{Imayoshi_1992in}, the holomorphic cotangent space at $\Xst$ of $\cT_{0,n}$ can be identified with the space of integrable holomorphic quadratic differentials on $\Xst$, which is denoted by $Q(\Xst)$.
In our situation, the following explicit description of $Q(\Xst)$ is convenient.

\begin{lemma}
  \label{lm:quad}
  Let $\varphi$ be a holomorphic function on $\Xst \subseteq \bC$. $\varphi \diff z \otimes \diff z$ is an integrable holomorphic quadratic differential on $\Xst$ if and only if $\varphi$ has the form
  \begin{equation*}
    \varphi = \sum^{n-2}_{k = 0}\frac{\rho_k}{z - z_k},
  \end{equation*}
  with the residues $\rho_k$ satisfying
  \begin{equation*}
    \sum^{n-2}_{k = 0} \rho_k = 0,\; \sum^{n-2}_{k = 0} \rho_kz_k = 0.
  \end{equation*}
\end{lemma}
Although in this paper, we fix $z_k$ to be $k$, $k=\aset{0,1,\cdots,n-2}$, we would like to point out that the above lemma is independent of the particular value of $z_k$ actually.
\begin{proof}
  The sufficiency part can be checked by direct calculation.
  We only show the necessity part.

  As $\varphi \diff z \otimes \diff z$ is integrable on $\Xst$, $\varphi$ is a meromorphic function on $\bC$ with at most simple poles at $z_k$, $k = 0,\cdots,n-2$.
  Therefore, taking $\rho_k$ to be the residue of $\varphi$ at $z_k$, we know that $\varphi -\sum^{n-2}_{k = 0}\frac{\rho_k}{z - z_k}$ is a holomorphic function on $\bC$.

  Let $w \coloneqq 1/z$.
  We have
  \begin{equation*}
    \varphi(z) \diff z \otimes \diff z = \frac{1}{w^4} \varphi(1/w) \diff w \otimes \diff w.
  \end{equation*}
  Using the integrability of $\varphi \diff z \otimes \diff z$ on $\Xst$ again, $\frac{1}{w^4}\varphi(1/w)$ has a simple pole at most at $0$.
Then Liouville's theorem implies that $\varphi -\sum^{n-2}_{k = 0}\frac{\rho_k}{z - z_k}$ must be zero.
  And the constraints on $\rho_k$ also follow from the integrability of $\varphi \diff z \otimes \diff z$ near $\infty$.
\end{proof}

From now on, we will identify $Q(\Xst)$ with the subspace of meromorphic functions on $\bC$ consisting of functions appearing in Lemma~\ref{lm:quad}.

Let $g$ be flat metric on $\Xst$ with cone angle $\alpha_k$ at $z_k$, $k = 0,\cdots,n-1$.
As we have said, $g$ must be equal to metric (\ref{eq:sp-metric}) up to a constant.
Therefore, we can write
\begin{equation*}
  g = e^{\gamma} \dzs
\end{equation*}
globally.
Similar to (\ref{eq:hp}), using the identification between the holomorphic cotangent space of $\cT_{0,n}$ at $\Xst$ and $Q(\Xst)$, we can define a Hermitian cometric for $\cT_{0,n}$ using $g$ in the following way
\begin{equation}
  \label{eq:come}
  h^*_{\pw}(\varphi, \psi) \coloneqq \Vol(X,g) \int_X \varphi\bar{\psi} e^{-\gamma} \dzs,
\end{equation}
where $\varphi,\psi \in Q(\Xst)$ and $\Vol(X,g) = \int_X e^{\gamma} \dzs$ is the volume of $g$.
Note that the cone angle assumption (\ref{eq:alpha-range}) guarantees that the integral in (\ref{eq:come}) is finite.

Using the cometric (\ref{eq:come}), we can see that for each $\psi \in Q(\Xst)$, $\bar{\psi} e^{-\gamma} \Vol(X,g)$ corresponds to an element in the dual of ${Q}(\Xst)$.
In other words, we have the following identification as in~\cite{Schumacher_2011we},
\begin{equation}
  \label{eq:t-id}
  \rT_{\Xst} \cT_{0,n} \simeq \aset{\bar{\psi} e^{-\gamma} \Vol(X,g) | \psi\in Q(\Xst)} \eqqcolon V(\Xst,g),
\end{equation}
where and in the following, we use the symbol $\rT_{\bullet}$ (resp.\ $\bar{\rT}_{\bullet}$) to denote the holomorphic (resp.\ anti-holomorphic), or $(1,0)$-type (resp.\ $(0,1)$-type), tangent space of a complex manifold at some point.
The holomorphic (anti-holomorphic) cotangent space is denoted by similar notations.

If $g$ is a complete hyperbolic metric on $X$, a well-known result is that a tangent vector of $\cT_{0,n}$ can be represented by a harmonic Beltrami differential with respect to $g$.
The above identification is an analog of this fact for singular Euclidean metrics.
Especially, one can check that an element in $V(\Xst,g)$ is also harmonic with respect to $g$.
However, compared with the classical case, there is a serious drawback for elements of $V(\Xst,g)$ from the analytical viewpoint.
Namely, in the classical case, a harmonic Beltrami differential lies in $\rL^{\infty}(\Xst, \rT \Xst \otimes \bar{\rT}^* \Xst)$, while in our case, an element of $V(\Xst, g)$ (also viewed as a section of $\rT \Xst \otimes \bar{\rT}^* \Xst$) is not $\rL^{\infty}$ in general by checking the singularities of $g$.\footnote{If $n=4$ and $\alpha_k = 1/2$ for $k = 0,1,2,3$, the element of $V(\Xst,g)$ is $\rL^{\infty}$ indeed. However, this special case is rather exceptional in many facets.}

Using (\ref{eq:t-id}), the cometric (\ref{eq:come}) induces a Hermitian metric on $\rT_{\Xst} \cT_{0,n}$,
\begin{equation}
  \label{eq:wp-me}
  h_{\pw}(\mu,\nu) \coloneqq \frac{1}{\tpi} \frac{\int_X \mu\bar{\nu} e^{\gamma} \diff z \wedge \diff \zb}{\Vol(X,g)},
\end{equation}
where $\mu,\nu \in V(\Xst,g) \simeq \rT_{\Xst} \cT_{0,n}$.
$h_{\pw}$ is a variant of the classical Weil-Petersson metric on $\cT_{0,n}$.
But for the sake of simplicity, we call the $2$-form associated with $h_{\pw}$ the \emph{Weil-Petersson form} (defined by $g$) directly and denote it by $\omega_{\pw}$.

A natural question about $\omega_{\pw}$ is whether it is a closed form, or equivalently, whether the Hermitian metric (\ref{eq:wp-me}) is K\"ahler.
If $g$ is a hyperbolic metric with conic singularities, Takhtajan and Zograf,~\cite{Takhtajan_2003hy}, have shown that a metric defined in the same way as (\ref{eq:wp-me}) is a \kah\ metric, also c.f.~\cite{Schumacher_2011we}.
In our case, we will deduce the closedness of $\omega_{\pw}$ from Theorem~\ref{thm:main}, i.e.\ Corollary~\ref{cor:main}.

\section{Some analytic results}
\label{sec:some-ana-res}
\subsection{}
In this section, we are going to show some analytic results used in the proof of Theorem~\ref{thm:main}.
These results are mainly about the properties of an integral operator associated with the Beltrami equation.

In the rest of the paper, we denote $\Delta_{\epsilon} \coloneqq \aset{z \in \bC| |z| < \epsilon}$ and $\Delta = \Delta_1$.

\subsection{The Beltrami equation.}
\label{sub:prep}
As a preparation, we review some facts about the Beltrami equation, i.e.\ the following equation,
\begin{equation}
  \label{eq:bel}
  \begin{cases}
    \pdv{w}{\zb} = \mu(z) \pdv{w}{z}, \\
    w(0) = 0,\; w(1) = 1,\; w(\infty) = \infty,
  \end{cases}
\end{equation}
in which $\mu \in \rL^{\infty}(\bC)$.
We need the following well-known properties of the Beltrami equation.
The readers can find a proof in~\cite[Ch.~4]{Imayoshi_1992in}.
\begin{enumerate}[label=\Alph*.]
\item If $\esup_{z\in\bC} |\mu(z)| < 1$, then there is a unique solution $w(z)$ to equation (\ref{eq:bel}).
  Moreover, such a solution $w(z)$ is a quasiconformal homeomorphism on $\hbC$.
\item Let $s\in \bC$ be a parameter with a small enough norm and $w_s$ be the solution of (\ref{eq:bel}) with $s\mu$ substituting $\mu$.
  Then for a fixed $z$, $w_s(z)$ is a holomorphic function of $s$ and the tangent vector induced by the derivative with respect to $s$ is
  \begin{equation}
    \label{eq:bel-diff}
    \pdv{w_s}{s}\Big|_{s = 0}(z) = \Big(\frac{1}{\tpi} \int_{\bC} \mu(\zeta) R(\zeta,z) \diff \zeta \wedge \diff \bar{\zeta}\Big) \frac{\partial}{\partial z},\; w_0(z) = z,
  \end{equation}
  where
  \begin{equation*}
    R(\zeta,z) \coloneqq \frac{z(z-1)}{\zeta(\zeta-1)(\zeta-z)} = \frac{1}{\zeta-z} + \frac{z-1}{\zeta} - \frac{z}{\zeta -1}.
  \end{equation*}
\item Let $\rL^{\infty}_1(\bC) \coloneqq \aset{\mu\in \rL^{\infty}(\bC) | \esup_{z\in \bC} |\mu(z)| < 1}$.
  As before, we choose $\Xst$ as the base point of $\cT_{0,n}$.
  Any point of $\cT_{0,n}$ can be represented by a solution of (\ref{eq:bel}) with respect to some element of $\rL^{\infty}_1(\bC)$.
  Indeed, any point of $\cT_{0,n}$ can be represented by a smooth solution of (\ref{eq:bel}) with a smooth $\mu$, see~\cite[Theorem~3]{Earle_1967re}.
\item Let $w_1,w_2$ be the solutions of (\ref{eq:bel}) with respect to $\mu_1,\mu_2 \in \rL^{\infty}_1(\bC)$ respectively.
  Then $w_1$ and $w_2$ represent the same point in $\cT_{0,n}$ if and only if $w_1$ is homotopic to $w_2$ through quasiconformal homeomorphisms relative to $z_k$, $k=0,\cdots,n-1$.
\end{enumerate}

\subsection{Properties of an integral operator.}
For our later usage, we need to investigate the property of the integral operator appearing in (\ref{eq:bel-diff}).
We begin with the following lemma.
\begin{lemma}
  \label{lm:int}
  For every $p$ with $2 < p < \infty$ and for every $h(z) \in \rL^p(\bC)$, the function $H(z)$ defined by
  \begin{equation}
    \label{eq:g-def}
    H(z) \coloneqq \frac{1}{\tpi} \int_{\bC} h(\zeta) R(\zeta,z) \diff \zeta \wedge \diff \bar{\zeta}
  \end{equation}
  is a uniformly H\"older continuous function on $\bC$, with exponent $1-2/p$ and satisfies
  \begin{equation}
    \label{eq:pb-cond}
    H(0) = 0,\; H(1) = 0,\; H(z) = \cO(|z|)\text{ as }z \rightarrow \infty.
  \end{equation}

  Moreover, $H(z)$ has $\rL^p_{\loc}$ distributional derivatives.
  In fact, $H(z)$ is the unique solution to the $\bar{\partial}$-equation for $h(z)$ satisfying the condition (\ref{eq:pb-cond}), i.e.\
  \begin{equation*}
    \pdv{H}{\zb} = h.
  \end{equation*}
\end{lemma}

\begin{proof}
  Define $\opP$ to be the following integral operator,
  \begin{equation*}
    \label{eq:def-opp}
    \opP(h) \coloneqq \frac{1}{\tpi} \int_{\bC} (\frac{1}{\zeta-z} - \frac{1}{\zeta})h(\zeta) \diff \zeta \wedge \diff \bar{\zeta}.
  \end{equation*}
  Using $\opP$, we can write $H$ as
  \begin{equation}
    \label{eq:decomp-H}
    H(z) = \opP(h)(z) + Kz,
  \end{equation}
  where $K$ is a constant depending on $h$.
  With such a decomposition, Lemma~\ref{lm:int} is a direct consequence of~\cite[Lemma~4.20 \& Proposition~4.23]{Imayoshi_1992in}.
As an example, we show the asymptotic property of $H(z)$ near the infinity.
  By~\cite[(4.10)]{Imayoshi_1992in}, there is a constant $K_p$ depending only on $p$ such that
  \begin{equation*}
    |\opP(h)(z)| \le K_p \norm{h}_p \cdot |z|^{1-2/p}.
  \end{equation*}
  By (\ref{eq:decomp-H}), the above estimate implies
  \begin{equation*}
    H(z) = \cO(|z|^{1-2/p}) + \cO(|z|) = \cO(|z|)\text{ as }z\rightarrow \infty.
  \end{equation*}
\end{proof}

To apply Lemma~\ref{lm:int}, we have to assume that $h(z)\in \rL^p(\bC)$, which does not always hold in some interesting cases.
For example, in (\ref{eq:bel-diff}), $\mu$ is only an $\rL^{\infty}$ function.
Nevertheless, we can generalize Lemma~\ref{lm:int} in the following way to weaken the assumption on the integrability.
\begin{lemma}
  \label{lm:int2}
  For every $p$ with $2 < p < \infty$ and for every $h(z)$ satisfying $h(z) \in \rL^p_{\loc}(\bC)$ and $h(1/z) \in \rL^p_{\loc}(\bC)$, the function $H(z)$ defined by (\ref{eq:g-def}) is still well defined.
  Furthermore, $H(z)$ satisfies all the properties in Lemma~\ref{lm:int} except that $H(z)$ is only locally $(1-2/p)$-H\"older continuous and has $\cO(|z|^{1+2/p})$ asymptotics near the infinity.
\end{lemma}
\begin{proof}
  We first check that $H$ is a well-defined locally H\"older continuous function and satisfies (\ref{eq:pb-cond}).

  Let $\chi_{\Delta}$ be the characteristic function of the open unit disk $\Delta$.
  We denote $h_1,h_2$ to be the function $h\chi_{\Delta},h(1-\chi_{\Delta})$ respectively.
  Therefore, we can write $H = H_1 + H_2$, where
  \begin{equation*}
    H_1(z) \coloneqq \frac{1}{\tpi} \int_{\Delta} h_1(\zeta) R(\zeta,z) \diff \zeta \wedge \diff \bar{\zeta},\; H_2(z) \coloneqq \frac{1}{\tpi} \int_{\bC- \Delta} h_2(\zeta) R(\zeta,z) \diff \zeta \wedge \diff \bar{\zeta}.
  \end{equation*}
  Due to Lemma~\ref{lm:int}, we know that $H_1(z)$ satisfies all the required properties.
  Thus, we only need to deal with $H_2(z)$.

  By changing variable, $\xi = 1/\zeta$, we have
  \begin{equation}
    \label{eq:h22}
    H_2(z) = \frac{-z^2}{\tpi} \int_{\Delta} h_2(1/\xi) \frac{\xi^2}{\bar{\xi}^2} R(\xi,1/z)  \diff \xi \wedge \diff \bar{\xi},\quad\text{for }z \neq 0,
  \end{equation}
  and $H_2(0) = 0$.
  Set $g(\xi) \coloneqq h_2(1/\xi) \xi^2/\bar{\xi}^2$.
  Then $g \in \rL^p(\Delta)$ due to the assumption on $h$.
  As a result, we can apply Lemma~\ref{lm:int} to $g$ and conclude that $H_2(z)$ is locally $(1-2/p)-$H\"older continuous on $\bC - \aset{0}$.
At the same time, for $z$ in a small neighborhood of the origin, we have
  \begin{equation}
    \label{eq:int2-1}
    |H_2(z) - H_2(0)| = |H_2(z)| \le C_1 |z|^2 \cdot \frac{1}{|z|} = C_1|z|,
  \end{equation}
  which implies that $H_2(z)$ is a continuous function on $\bC$.
  Here and in the following, the estimate constants $C_i$ may depend on $\norm{g}_{\rL^p(\Delta)}$.

  The next step is to show that $H_2$ is also H\"older continuous near a neighborhood of the origin and has the desired asymptotic property.
  We need some calculations.
Choose $w_m$ with $0< |w_m| \le \epsilon < 1$, $m = 1,2$.
Therefore, we have
  \begin{equation}
    \label{eq:est-r1}
    \begin{aligned}
      & \Big|\int_{\Delta} g(\xi) \frac{w_1}{\xi - w_1^{-1}} \dxs - \int_{\Delta} g(\xi) \frac{w_2}{\xi - w_2^{-1}} \dxs\Big| \\
      \le & |w_1 - w_2| \int_{\Delta} |g(\xi)| \frac{|w_1w_2\xi - w_1 - w_2|}{(1 - |w_1\xi|) (1- |w_2\xi|)} \dxs \\
      \le & |w_1 - w_2| \int_{\Delta} |g(\xi)| \frac{\epsilon^2|\xi| + 2\epsilon}{(1- \epsilon |\xi|)^2} \dxs \le C_2 |w_1 - w_2|.
    \end{aligned}
  \end{equation}
  The calculation like (\ref{eq:est-r1}) also holds for the other two terms in $R(\xi,1/z)$.
  Hence,
  \begin{equation}
    \label{eq:est-r}
    \Big|w_1\int_{\Delta} g(\xi) R(\xi,1/w_1)  \dxs - w_2\int_{\Delta} g(\xi) R(\xi,1/w_2)  \dxs\Big| \le C_3 |w_1 - w_2|.
  \end{equation}
  Meanwhile, by Lemma~\ref{lm:int}, for $0< |z| \le \epsilon$, we have
  \begin{equation}
    \label{eq:est-r2}
    \Big|z\int_{\Delta} g(\xi) R(\xi,1/z)  \dxs\Big| \le C_4.
  \end{equation}
  By (\ref{eq:est-r}) and (\ref{eq:est-r2}),
  \begin{equation}
    \label{eq:int2-2}
    |H_2(w_1) - H_2(w_2)| \le (\epsilon C_3 + C_4)|w_1 - w_2|.
  \end{equation}
  Combining (\ref{eq:int2-1}) and (\ref{eq:int2-2}), we know that $H_2(z)$ is Lipschitz continuous, H\"older continuous especially, near the origin.

  Once $H(z)$ is well defined, it is trivial to see that $H(z)$ takes the desired values at $0,1$.
  To show that $H(z)$ has the desired asymptotic property, we also only need to show that the same result holds for $H_2(z)$.
  For this purpose, we notice the following estimate.
  Due to ($1-2/p$)-H\"older continuity in Lemma~\ref{lm:int} or~\cite[Lemma~4.20]{Imayoshi_1992in}, there exists a constant $C_5$ such that
  \begin{equation*}
    \Big| z^2 \int_{\Delta} g(\xi) (\frac{1}{\xi - z^{-1}} - \frac{1}{\xi}) \dxs \Big| \le |z|^2 \cdot C_5 |z^{-1} - 0|^{1- 2/p} = C_5 |z|^{1+ 2/p}.
\end{equation*}
Meanwhile, by direct calculation,
  \begin{equation*}
    \Big| z \int_{\Delta} g(\xi) (\frac{1}{\xi} - \frac{1}{\xi-1}) \dxs \Big| \le C_6 |z|.
  \end{equation*}
  Using the definition of $R$ again, the above two estimates imply that $H_2(z) = \cO(|z|^{1+2/p})$ as $z \rightarrow \infty$.

  Now, we show that $H(z)$ has the $\rL^p_{\loc}$ distributional derivatives.
  Again, we only need to show this property for $H_2(z)$.
  Near any $z \neq 0$, taking a small neighborhood $U$ of $z$, Lemma~\ref{lm:int} implies that
  \begin{equation}
    \label{eq:zb-g}
    \pdv*{}{\zb}{\Big(\int_{\Delta} g(\xi)R(\xi,1/z)  \dxs\Big)} =- \frac{1}{\zb^2} g(1/z)
  \end{equation}
  holds as distributions over $\sC^{\infty}_c(U)$, where we define $g(\xi) \equiv 0$ for $\xi \in \bC - \Delta$ for r.h.s.\ of the equality.
  By (\ref{eq:h22}) and (\ref{eq:zb-g}), we know that on $\bC - \aset{0}$, the distributional $\bar{\partial}$ derivative of $H_2(z)$ is $\rL^p_{\loc}$.
  Using the same argument, we can also show that on $\bC - \aset{0}$, the distributional $\partial$ derivative of $H_2(z)$ is $\rL^p_{\loc}$, too.

  On the other hand, using (\ref{eq:h22}), we can write $H_2(z)$ as
  \begin{equation*}
    H_2(z) = \int_{\Delta} g(\xi) \Bigl(\frac{z^3}{z\xi - 1} - \frac{z^2}{\xi}\Bigr) \dxs + z\int_{\Delta} g(\xi)\Big(\frac{1}{\xi} - \frac{1}{\xi-1}\Big) \dxs.
  \end{equation*}
  Using the above expression, we find that if $|z| \le \epsilon < 1$, $H(z)$ has the classical derivatives and
  \begin{equation}
    \label{eq:ph2}
    \pdv{H_2}{\zb} = \int_{\Delta} g(\xi) \pdv*{}{\zb}{\Big(\frac{z^3}{z\xi - 1} - \frac{z^2}{\xi}\Big)}\dxs = 0.
  \end{equation}
  As a result, Weyl's lemma implies that $H_2(z)$ is smooth (holomorphic actually) near the origin, where the derivatives are $\rL^p_{\loc}$ certainly.

  Note that the calculation similar to (\ref{eq:ph2}) also implies that (\ref{eq:zb-g}) holds as distributions over $\sC^{\infty}_c(\bC)$.
  Then, using (\ref{eq:zb-g}), we can check that $H(z)$ satisfies the $\bar{\partial}$-equation for $h(z)$ by a direct calculation.
\end{proof}

As an application of the above two lemmas, we discuss a property about the elements in $V(\Xst,g)$.

\begin{proposition}
  \label{prop:two-cl}
  Let $a(z) \in V(\Xst,g)$.
  \begin{enumerate}[label=(\arabic*)]
  \item\label{prop:tc1} For any $p> 2$ satisfying
    \begin{equation}
      \label{eq:p-cond}
      p(1-2\alpha_k) < 2,\;k=0,\cdots,n-1,
    \end{equation}
    $a(z) \in \rL^p_{\loc}(\bC)$ and $a(1/z) \in \rL^p_{\loc}(\bC)$.
  \item Due to the integrability of $a(z)$, we can define
    \begin{equation*}
      Y(z) \coloneqq \frac{1}{\tpi} \int_{\bC} a(\zeta) R(\zeta,z) \diff \zeta \wedge \diff \bar{\zeta}.
    \end{equation*}
    Then $Y(z)$ is a continuous function on $\bC$ and smooth away from $z_k$, $k = 0, \cdots, n-2$.
  \item Let $k = 0, \cdots, n-2$.
    If $0< \alpha_k < 1/2$, near $z_k$, $Y(z)$ is $(2\alpha_k - \varepsilon)$-H\"older continuous for any $\varepsilon>0$ sufficiently small.
    If $1/2\le \alpha_k < 1$, near $z_k$, $Y(z)$ is $(1 - \varepsilon)$-H\"older continuous for any $\varepsilon>0$ sufficiently small.
  \end{enumerate}
\end{proposition}

\begin{proof}
  (1) By (\ref{eq:sp-metric}) and (\ref{eq:t-id}), $a(z)$ can be written as $\bar{\psi} e^{-\gamma} = \bar{\psi} \prod^{n-2}_{k = 0} |z - z_k|^{2\alpha_k}$ for some $\psi \in Q(\Xst)$.
  Then Lemma~\ref{lm:quad} implies that $a(z)$ is smooth away from $z_k$, $k = 0, \cdots, n-2$, and near the singularity $z_k$, $k = 0,\cdots,n-2$, $a(z)$ can be estimated in the following way,
  \begin{equation}
    \label{eq:asym-k}
    a(z) = \cO(\frac{1}{|z - z_k|^{1- 2\alpha_k}}).
  \end{equation}
  Moreover, using (\ref{eq:gb-cond}), near $z_{n-1} = \infty$, we have
  \begin{equation*}
    a(z) = \cO(|z|^{1 - 2\alpha_{n-1}}).
  \end{equation*}
  Using the above two estimates, when $p$ satisfies (\ref{eq:p-cond}), the direct calculation shows that $a(z)$ has the desired integrability.

  (2) By Lemma~\ref{lm:int2}, we know that $Y(z)$ is a continuous function on $\bC$ and $\pdv{Y}{z} = a$.
  Since $a(z)$ is smooth away from $z_k$, $k = 0, \cdots, n-2$, due to the elliptic regularity of the $\bar{\partial}$ operator, the same is true for $Y(z)$.

  (3) This part results from the Sobolev embedding theorem essentially.
  Here, we give a direct proof using the results in this section.
  We first assume that $0< \alpha_k < 1/2$.
  Choose $D_k$ to be a small neighborhood of $z_k$ and $\chi_{D_k}(z)$ to be the characteristic function of $D_k$.
  Let $a_k = \chi_{D_k}(z)a(z)$.
  Due to (\ref{eq:asym-k}), $a_k\in \rL^{2/(1-2\alpha_k + \varepsilon)}(\bC)$ for any sufficient small $\varepsilon >0$.
  As a result, by Lemma~\ref{lm:int},
  \begin{equation*}
    Y_k(z) \coloneqq \frac{1}{\tpi} \int_{\bC} a_k(\zeta) R(\zeta,z) \diff \zeta \wedge \diff \bar{\zeta}
  \end{equation*}
  is a $(2\alpha_k - \varepsilon)$-H\"older continuous function on $\bC$.
  Meanwhile, as distributions, $Y$ and $Y_k$ satisfy
  \begin{equation*}
    \pdv{(Y- Y_k)}{\zb} = a - a_k.
  \end{equation*}
  By the definition of $a_k$ and Weyl's lemma, we know that $Y-Y_k$ is a smooth function near $z_k$, which means that $Y$ is also $(2\alpha_k - \varepsilon)$-H\"older continuous near $z_k$.

  For the $1/2\le \alpha_k < 1$ case, also by (\ref{eq:asym-k}), we notice that $a_k \in \rL^{N}(\bC)$ for any $N> 2$.
  Then, using the same arguments as above, we can obtain the regularity of $Y(z)$ near $z_k$.
\end{proof}

\begin{remark}
  \label{rk:dy}
  (1) By our assumption on the range of of $\alpha_k$, (\ref{eq:alpha-range}), the desired $p>2$ satisfying (\ref{eq:p-cond}) always exists.

  (2) Since $Y(z)$ is smooth away from $z_k$, $k = 0, \cdots, n-2$, the $\partial$ distributional derivative of $Y(z)$ can be calculated as
  \begin{equation*}
    \pdv{Y(z)}{z},\quad z\in X,
  \end{equation*}
  where the partial derivatives appearing in the above formula are the classical partial derivatives.
  This result is a simple corollary of the $\rL^p_{\loc}$ integrability of the distribution derivatives of $Y(z)$.
  A similar result also holds for the $\bar{\partial}$ distributional derivative of $Y(z)$.

  (3) One can also formulate Proposition~\ref{prop:two-cl} in terms of the vector field $Y \pdv*{}{z}$.
  In fact, if we define the value of this vector field to be $0$ at $z_{n-1} = \infty$, by using Lemma~\ref{lm:int} \&~\ref{lm:int2} as in the above proof, this vector field also has a similar regularity property near $z_{n-1}$ as other $z_k$, $0\le k \le n-2$.
\end{remark}

\subsection{A result about the Hilbert transformation.}
\label{sub:asym}
Let $h(z) \in \rL^p(\bC)$, $2 < p < \infty$.
Recall that we define $\opP(h)$ in (\ref{eq:def-opp}).
By~\cite[Proposition~4.23]{Imayoshi_1992in}, the distributional $\partial$ derivative of $\opP(h)$ is
\begin{equation*}
  \pdv{\opP(h)}{z} = \frac{1}{\tpi} \int_{\bC} \frac{h(\zeta)}{(\zeta-z)^2} \diff \zeta \wedge \diff \bar{\zeta} \eqqcolon \opT(h)(z).
\end{equation*}
Note that the integral in the middle is not a usual Lebesgue integral but the Hilbert transformation of $h(z)$, that is, a singular integral.
To give it a meaning, we first assume $h\in \sC^{\infty}_c(\bC)$ and calculate the improper integral using the Cauchy principal value.
For a general $h$, we approximate it by $\aset{h_n}\subseteq \sC^{\infty}_c(\bC)$ in $\rL^p(\bC)$ and define
\begin{equation*}
  \opT(h) \coloneqq \lim_{n\rightarrow \infty}\opT(h_n)
\end{equation*}
in $\rL^p(\bC)$ using the Calder{\'o}n-Zygmund theorem.
For details, readers can see~\cite[\S~4.2]{Imayoshi_1992in}.

For our later usage, we discuss an asymptotic property of the Hilbert transformation.
\begin{lemma}
  \label{lm:asym}
  Let $g(z) \in \rL^{p}(\Delta)$, $p > 2$, whose distributional derivatives are integrable on $\Delta$.
  We suppose that there exists $\beta < 1$ such that\footnote{For a measurable function, the Big O notation means that the needed estimates hold outside a measure zero set.} $g(z) = \cO(1/|z|^\beta)$ and $\diff g (z) = \cO(1/|z|^{\beta+1})$ as $z \rightarrow 0$.
  Then the Hilbert transformation of $g(z)$,
  \begin{equation*}
    G(z) \coloneqq \frac{1}{\tpi} \int_{\Delta} \frac{g(\xi)}{(\xi - z)^2}  \diff \xi \wedge \diff \bar{\xi},
  \end{equation*}
  satisfies
  \begin{equation}
    \label{eq:g-asym}
    G(z) = \cO(1/|z|^\gamma) \text{ as } z\rightarrow 0,
  \end{equation}
  where $\gamma = \max\aset{2/p,\beta} < 1$.
\end{lemma}

\begin{proof}
To show the estimate (\ref{eq:g-asym}), we first notice that we can assume that
\begin{equation}
  \label{eq:est-g}
  |g(z)| \le \frac{C}{|z|^\beta},\; |\diff g(z)| \le \frac{C}{|z|^{\beta+1}}
\end{equation}
hold almost everywhere on $\Delta$.
In fact, the conditions in Lemma~\ref{lm:asym} imply that the above estimates hold at least on a small disk $\Delta_{\epsilon}$.
Let $\tau$ be a cut-off function such that $\tau(\Delta_{\epsilon/3}) = 1$ and $\tau(\Delta - \Delta_{2\epsilon/3}) = 0$.
Write $g$ as $g = \tau g + (1-\tau)g$.
Since $(1-\tau)g$ vanishes near the origin, when $|z|$ small enough, the improper integral appearing in $\opT((1-\tau)g)(z)$ is just a normal Lebesgue integral.
Therefore, $\opT((1-\tau)g)(z)$ is smooth near the origin and the desired estimate (\ref{eq:g-asym}) for $(1-\tau)g$ follows.
As a result, to prove (\ref{eq:g-asym}) for $g$, we only need to prove the same result for $\tau g$, which satisfies (\ref{eq:est-g}) on $\Delta$.
An additional advantage of using such a cut-off function is that we can and will assume that $\supp(g)$ is sufficiently small.

Let $\chi(z) \in \sC^{\infty}_c(\Delta)$ be a bump function supported around the origin which satisfies
\begin{equation}
  \label{eq:def-chi}
  0\le \chi(z) \le 1,\; \int_{\Delta} \chi(z) \dzs =1.
\end{equation}
For $n\in \bN_+$, we denote $\chi(nz)n^2$ by $\chi_n(z)$.
As a result, $\int_{\Delta} \chi_n(z) \dzs = 1$.
To calculate the Hilbert transformation of $g$, we need to approximate $g$ using smooth functions.
Here, we choose to construct such smooth approximations using the convolution with $\chi_n$, that is,
\begin{equation}
  \label{eq:def-gn}
  g_n(z) \coloneqq (g * \chi_n)(z) = \int_{\bC} g(z - u) \chi_n(u) \lvert\diff u\rvert^2.
\end{equation}
By~\cite[Theorem~1.3.2]{Hormander_1990ab}, $\aset{g_n}$ converges to $g$ in the $\rL^p$ norm.
To make our ongoing arguments run more smoothly, we will only consider $n$ large enough such that
\begin{equation}
  \label{eq:ass-gn}
  \supp(g_n) \subseteq \Delta_{1/3},\; \Vert g_n\Vert_{\rL^p(\Delta)} \le 2 \norm{g}_{\rL^p(\Delta)}.
\end{equation}

To show Lemma~\ref{lm:asym}, we proceed in three steps.
Firstly, we show that estimates like (\ref{eq:est-g}) hold uniformly for $g_n$.
Secondly, we show that the estimate like (\ref{eq:g-asym}) for $\opT(g_n)$ holds independent of $n$.
The last step is to deduce the desired estimate for $g$ from that of $g_n$.

\vspace{0.5\baselineskip}
\noindent\uline{\textsc{Step 1.}} We are going to show that there exists a constant $C_0$, independent of $n$, such that the following estimates hold on $\Delta - \aset{0}$,
\begin{equation}
  \label{eq:est-gn}
  |g_n(z)| \le \frac{C_0}{|z|^\beta},\; |\diff g_n(z)| \le \frac{C_0}{|z|^{\beta+1}}.
\end{equation}

By the general property of distributions, c.f.~\cite[Theorem~4.1.1]{Hormander_1990ab}, the following equalities hold,
\begin{equation*}
  \pdv{g_n}{z} = \pdv{g}{z} * \chi_n,\; \pdv{g_n}{\zb} = \pdv{g}{\zb} * \chi_n.
\end{equation*}
Therefore, we only need to show the estimate for $g_n$ in~(\ref{eq:est-gn}) and the remaining estimate for $\diff g_n$ can be proved in the same way with $\beta$ replaced by $\beta+1$.

Using (\ref{eq:est-g}), (\ref{eq:def-chi}), (\ref{eq:def-gn}) and by changing the variable, for $z \neq 0$, we have
\begin{multline}
  \label{eq:abs-gn}
  |g_n(z)| = \Big| n^2\int_{\Delta_{1/n}} g(z - u) \chi(nu) \lvert\diff u\rvert^2 \Big| \\
  \le Cn^2\int_{\Delta_{1/n}} \frac{\lvert\diff u\rvert^2}{|z - u|^\beta}
  = Cn^2 |z|^{2-\beta} \int_{\Delta_{1/(n|z|)}} \frac{\lvert\diff w\rvert^2}{|1 - w|^\beta}.
\end{multline}
To estimate the r.h.s.\ of (\ref{eq:abs-gn}), we discuss two cases.
If $1/(n|z|) \le 1/2$, (\ref{eq:abs-gn}) implies
\begin{multline}
  \label{eq:abs-gn1}
  |g_n(z)| \le Cn^2 |z|^{2-\beta} \int_{\Delta_{1/(n|z|)}} \frac{\lvert\diff w\rvert^2}{|1 - |w||^\beta} \le C2^\beta n^2 |z|^{2-\beta} \int_{\Delta_{1/(n|z|)}} \lvert\diff w\rvert^2 \\
  \le C_1 n^2 |z|^{2-\beta} \frac{1}{n^2 |z|^2} = \frac{C_1}{|z|^\beta}.
\end{multline}
If $1/(n|z|) > 1/2$, by taking
\begin{equation*}
  D_{1,1/(n|z|)} \coloneqq \aset{z \in \bC| |z - 1| \le 1/(n|z|)},
\end{equation*}
we can see that $D_{1,1/(n|z|)} \subseteq \Delta_{3/(n|z|)}$.
Thus, we have
\begin{equation*}
  \int_{\Delta_{1/(n|z|)}} \frac{\lvert\diff w\rvert^2}{|1 - w|^\beta} = \int_{D_{1,1/(n|z|)}} \frac{\lvert\diff w\rvert^2}{|w|^\beta} \le \int_{\Delta_{3/(n|z|)}} \frac{\lvert\diff w\rvert^2}{|w|^\beta} = \frac{C_2}{n^{2-\beta}|z|^{2-\beta}}.
\end{equation*}
Note that $1/(n|z|) > 1/2$ is equivalent to $n < 2/|z|$.
Then by the above inequality and (\ref{eq:abs-gn}), we know
\begin{equation}
  \label{eq:abs-gn2}
  |g_n(z)| \le Cn^2 |z|^{2-\beta} \frac{C_2}{n^{2-\beta}|z|^{2-\beta}} = CC_2 n^\beta \le \frac{2^\beta CC_2}{|z|^\beta}.
\end{equation}
The desired estimate for $g_n$ in (\ref{eq:est-gn}) follows from (\ref{eq:abs-gn1}) and (\ref{eq:abs-gn2}) directly.

\vspace{0.5\baselineskip}
\noindent\uline{\textsc{Step 2.}} We denote the Hilbert transformation of $g_n$ to be
\begin{equation*}
  G_n(z) \coloneqq \opT(g_n)(z) =  \frac{1}{\tpi} \int_{\Delta} \frac{g_n(\xi)}{(\xi - z)^2}  \diff \xi \wedge \diff \bar{\xi}.
\end{equation*}
We are going to show that there exists a constant $C_3$, independent of $n$, such that
\begin{equation}
  \label{eq:est-biggn}
  |G_n(z)| \le \frac{C_3}{|z|^\gamma}
\end{equation}
holds for $0 < |z|\le 1/3$.

As we have said, to define the Hilbert transformation of $g_n$, we need to use the Cauchy principal value.
More precisely, $G_n$ has the following expression,
\begin{equation}
  \label{eq:def-biggn}
  G_n(z) = \frac{1}{\tpi} \int_{\Delta - D_{z,|z|/2}} \frac{g_n(\xi)}{(\xi - z)^2}  \diff \xi \wedge \diff \bar{\xi}
  + \frac{1}{\tpi} \mathrm{p.v.}\int_{D_{z,|z|/2}} \frac{g_n(\xi)}{(\xi - z)^2}  \diff \xi \wedge \diff \bar{\xi},
\end{equation}
where $D_{z,|z|/2}$ is a disk centered at $z$ with radius $|z|/2$.

We can calculate the principal value appearing in (\ref{eq:def-biggn}) as follows
\begin{equation*}
  \frac{1}{\tpi} \mathrm{p.v.}\int_{D_{z,|z|/2}} \frac{g_n(\xi)}{(\xi - z)^2}  \diff \xi \wedge \diff \bar{\xi} = \frac{1}{\tpi} \int_{D_{z,|z|/2}} \frac{g_n(\xi) - g_n(z)}{(\xi - z)^2}  \diff \xi \wedge \diff \bar{\xi}.
\end{equation*}
By the mean value theorem and (\ref{eq:est-gn}), for $\xi \in D_{z,|z|/2}$, we have
\begin{equation*}
  | \frac{g_n(\xi) - g_n(z)}{\xi - z} | \le \frac{C_4}{|z|^{\beta +1}},
\end{equation*}
where $C_4$ is independent of $n$.
Thus, the above two formulae yield
\begin{multline}
  \label{eq:est-biggn1}
  \Big|\frac{1}{\tpi} \mathrm{p.v.}\int_{D_{z,|z|/2}} \frac{g_n(\xi)}{(\xi - z)^2}  \diff \xi \wedge \diff \bar{\xi}\Big| \le \frac{C_4}{|z|^{\beta+1}} \int_{D_{z,|z|/2}} \frac{\dxs}{|\xi - z|}\\
  = \frac{C_4}{|z|^{\beta+1}} \int_{\Delta_{|z|/2}} \frac{\dxs}{|\xi|} = \frac{C_4}{2 |z|^\beta}.
\end{multline}

To control the \nth{1} term on the r.h.s.\ of (\ref{eq:def-biggn}), we note that if $z \neq 0$,
\begin{equation}
  \label{eq:est-biggn3}
  (\int_{\bC} \frac{\dxs}{|\xi(\xi - z)|^q})^{1/q} = \frac{1}{|z|^{2/p}} (\int_{\bC} \frac{\dxs}{|\xi(\xi - 1)|^q})^{1/q}  = \frac{C_5}{|z|^{2/p}},
\end{equation}
where $q = p/(p-1)$.
By (\ref{eq:ass-gn}) and (\ref{eq:est-biggn3}), the following estimate holds,
\begin{multline}
  \label{eq:est-biggn2}
  \Big| \int_{\Delta - D_{z,|z|/2}} \frac{g_n(\xi)}{\xi(\xi - z)}  \dxs\Big|
  \le \int_{\Delta}  \frac{|g_n(\xi)|}{|\xi(\xi - z)|} \dxs\\
\le C_5 \norm{g_n}_{\rL^p(\Delta)} \frac{1}{|z|^{2/p}} \le 2C_5 \norm{g}_{\rL^p(\Delta)} \frac{1}{|z|^{2/p}}
\end{multline}
Meanwhile, since for $\xi\in \Delta - D_{z,|z|/2}$, we have $|\xi - z| > |z|/2$, the following estimate also holds for $z \neq 0$,
\begin{equation}
  \label{eq:est-biggn2.5}
  \begin{aligned}
    &\big| \int_{\Delta - D_{z,|z|/2}} g_n(\xi) (\frac{1}{(\xi - z)^2} -\frac{1}{\xi(\xi - z)}) \dxs\big|\\
    = &\big| \int_{\Delta - D_{z,|z|/2}} g_n(\xi) (\frac{z}{\xi - z} \cdot \frac{1}{\xi(\xi - z)}) \dxs\big|\\
    \le & 2 \int_{\Delta - D_{z,|z|/2}} \frac{|g_n(\xi)|}{|\xi(\xi - z)|}\dxs
    \le 4 C_5\norm{g}_{\rL^p(\Delta)} \frac{1}{|z|^{2/p}},
  \end{aligned}
\end{equation}
where for the last inequality, we use (\ref{eq:est-biggn3}) again.

Now, combining (\ref{eq:def-biggn}), (\ref{eq:est-biggn1}), (\ref{eq:est-biggn2}) and (\ref{eq:est-biggn2.5}), we get (\ref{eq:est-biggn}) immediately.

\vspace{0.5\baselineskip}
\noindent\uline{\textsc{Step 3.}}
Since $\aset{g_n}$ converges to $g$ in the $\rL^p$ norm as $n\rightarrow+\infty$, by the Calder{\'o}n-Zygmund theorem, c.f.~\cite[Proposition~4.22]{Imayoshi_1992in}, $\aset{G_n}$ also converges to $G$ in the $\rL^p$ norm as $n \rightarrow +\infty$.
Therefore, we know that $\aset{G_n}$ converges to $G$ in measure, which implies that there exists a subsequence of $\aset{G_n}$ converges to $G$ almost everywhere.
Now, we deduce from (\ref{eq:est-biggn}) that
\begin{equation*}
  |G(z)| \le \frac{C_3}{|z|^\gamma}
\end{equation*}
holds almost everywhere on $\Delta_{1/3}$ and the proof of Lemma~\ref{lm:asym} is completed.
\end{proof}

\section{A proof of the main theorem}
\label{sec:proof-of-main}

\subsection{}
In this section, we finish the proof of Theorem~\ref{thm:main}.
To do this, we will borrow an argument due to Tian~\cite{Tian_1987aa}.
As we have recalled in Introduction, Tian uses this argument to show a result about the Weil-Petersson metric for a family of Calabi-Yau manifolds, which is rather similar to Theorem~\ref{thm:main}.
Thus, we will first recall Tian's result briefly and then turn to the proof of Theorem~\ref{thm:main}.

\subsection{Tian's result.}
Let $(M,\omega)$ be a Calabi-Yau closed manifold of complex dimension $n$ with the \kah\ form $\omega$, $[\omega] \in \rH^2(M,\bZ)$.
Denote the universal polarized deformation space of $M$ by $\cM_{\cy}$.
$\cM_{\cy}$ is smooth due to the Bogomolov-Tian-Todorov theorem.
The tangent space of $\cM_{\cy}$ at $M$ is isomorphic to $\rH^1(M, \rT M)_{\omega}$, a subspace of $\rH^1(M, \rT M)$.
Let $\varphi\in \rH^1(M, \rT M)_{\omega}$ (and choose a harmonic representative of $\varphi$).
By using the \kah\ metric on $M$, one can define the Weil-Petersson metric $\omega_{\pw}(\varphi,\bar{\varphi})$ similar to (\ref{eq:wp-me}).

On the other hand, let
\begin{equation*}
  D' \coloneqq \aset{\text{complex lines $l$ in }\rH^n(M,\bC)| \forall \psi\in l - \{0\}, Q(\psi,\bar{\psi}) > 0, Q(\psi,\psi) =0},
\end{equation*}
where $Q(-,-)$ is the intersection from on $\rH^n(M,\bC)$.
Then $D'$ inherits a \kah\ metric $\omega_{\rH}$ from $Q$.
There is a map $\Omega'$ from a small neighborhood of the origin in $\rH^1(M, \rT M)_{\omega}$ to $D'$ called the period mapping.
Then, Tian proves the following equality~\cite[Theorem~2]{Tian_1987aa},
\begin{equation}
  \label{eq:tian-eq}
  \omega_{\pw} = {\Omega'}^{*}(\omega_{\rH}).
\end{equation}

To show such a result, the key is Tian's following observation about the period mapping.
Let $s$ be a small complex parameter.
Then there exists a family of $n$-forms $\Psi_s$ on $M$ such that $\Omega'(s\varphi) = [\Psi_s] \in D'$.
Moreover,
\begin{equation}
  \label{eq:tian-lm}
  \pdv{\Psi_s}{s}\Big|_{s = 0} = c \Psi_0 + \varphi \mathbin\lrcorner \Psi_0,
\end{equation}
where $c_0$ is a constant and $\lrcorner$ is the contraction between two tensors.

Although our description is very sketchy, it is still not difficult to catch the similarity between the definitions of ${\Omega'}^{*}(\omega_{\rH})$ and  $\omega_{\tv}$.
As we have said in Introduction, Theorem~\ref{thm:main} is just a variant of Tian's equality (\ref{eq:tian-eq}) for the singular metrics in this perspective.
Consequently, it is natural to use the same argument to prove Theorem~\ref{thm:main}.
More precisely, it means that we need to find a family of forms similar to $\Psi_s$ and prove (\ref{eq:tian-lm}) for such forms.
It turns out such a family of forms can be constructed using $\eta(\bfu)$ defined in (\ref{eq:balpha}) more or less.
But, the difficulty is also obvious, the inevitable singularities near the punctures.
To overcome them, results in Section~\ref{sec:some-ana-res} come into play.

\subsection{A key lemma.}
In the rest of the paper, to apply the regularity results in Section~\ref{sec:some-ana-res}, we fix $p>2$ to be a number satisfies (\ref{eq:p-cond}), which is possible, due to the assumption on the range of cone angles (\ref{eq:alpha-range}), as we have noted in Remark~\ref{rk:dy}.

To prove Theorem~\ref{thm:main}, we need to show the following equality for any $v \in \rT_{\Xst} \cT_{0,n}$,
\begin{equation}
  \label{eq:key-eq}
  \omega_{\tv}(v,\bar{v}) = \omega_{\pw}(v,\bar{v}).
\end{equation}
As we have outlined, in this subsection, we will construct a family of forms (classes) associated with $v$ and prove a key lemma, i.e.\ Lemma~\ref{lm:key}, similar to (\ref{eq:tian-lm}).

Let $a \in V(\Xst,g)$ be the element corresponding to $v$ by (\ref{eq:t-id}).
Meanwhile, since $\rT_{\Xst}\cT_{0,n}$ is also isomorphic to a quotient space of $\rL^{\infty}(\Xst, \rT \Xst \otimes \bar{\rT}^* \Xst)$, we can find $\mu \in \rL^{\infty}(\Xst, \rT \Xst \otimes \bar{\rT}^* \Xst)$ whose equivalence class also corresponds to $v$.
In the following, we will view $\mu$ as an $\rL^{\infty}$ function on $\bC$.
Recall that $Q(\Xst)$ is the holomorphic cotangent space at $\Xst$ of $\cT_{0,n}$.
For any $\varphi \in Q(\Xst)$, since $a$ and $\mu$ are two different representative of $v$,  we have
\begin{equation}
  \label{eq:two-v}
  \langle \varphi, v \rangle = \int_X \varphi a \dzs = \int_X \varphi\mu \dzs,
\end{equation}
where $\langle-,-\rangle$ is the natural pair between the tangent space and the cotangent space.

Let $s\in \Delta_{\epsilon}$ be a small enough complex parameter.
By the result that we have recalled in \S~\ref{sub:prep}, we can find a solution $w_s$ to (\ref{eq:bel}) with $\mu$ replaced by $s\mu$.
Then we have the following map,
\begin{equation*}
  \begin{array}{cccc}
    \opc: &\Delta_{\epsilon} &\rightarrow& \cW_n \\
          &s    &\mapsto& (0, 1, w_s(z_2),\cdots,w_s(z_{n-1}))
  \end{array}.
\end{equation*}
As $\opc$ is holomorphic, by pulling back $\rP_{\cW_n}$, $\rP_{\Delta_{\epsilon}} \coloneqq \opc^*\rP_{\cW_n}$ gives a family of punctured spheres over $\Delta_{\epsilon}$.
When $s = 0$, $w_0 = \id$, which means that
\begin{equation*}
  \opc(0) = (0, 1, z_2, \cdots, z_{n-1})
\end{equation*}
is the base point $o$ of $\cW_n$.
Therefore, the fiber of $\rP_{\Delta_{\epsilon}}$ over the origin $0\in \Delta$ is $\rP_o$, which is biholomorphic to $\Xst$.

We remark that the above construction for the family $\rP_{\Delta_{\epsilon}}$ can be viewed as an application of the universal curve over the space of the Beltrami differentials.
Since there is more than one way to construct such a universal curve, it is possible to construct $\rP_{\Delta_{\epsilon}}$ by other methods.
Readers can find more materials about this in~\cite[\S~4.8]{Hubbard_2006te}.

As in \S~\ref{sub:def-tv}, we take $\bfu = (u_2,\cdots,u_{n-2})$ to be the global coordinates on $\cW_n$ induced from $\bC^{n-3}$ and identify $\bfu$ as a local coordinates near $\Xst$ in $\tnn$.
Then, in this local coordinate, we have
\begin{equation}
  \label{eq:v-loc}
  v = \sum_{i=2}^{n-2} \pdv{w_s(z_i)}{s}\Big|_{s=0} \frac{\partial}{\partial u_i}.
\end{equation}

Recall that with $\bfu$ fixed, $\eta(\bfu)$ defined in (\ref{eq:balpha}) is a $1$-form on $\rP_{\bfu} = X_{\bfu} = \bC - \aset{0,1,u_2,\cdots,u_{n-2}}$, which enters the definition of $\omega_{\tv}$.
In the following, we use $\eta_s$ as a shorthand for $\eta(w_s(z_2),\cdots,w_s(z_{n-2}))$.
Using (\ref{eq:v-loc}) and the definition of $\omega_{\tv}$, (\ref{eq:tv-local}), we can calculate the l.h.s.\ of (\ref{eq:key-eq}) as follows,
\begin{multline}
  \label{eq:l-key-eq}
  -i\omega_{\tv}(v,\bar{v}) = \pdv*{}{s,\bar{s}}{\Big(\log{\int_{X_s} \frac{\eta_s \wedge \overline{\eta}_s}{-\tpi}}\Big)}\Big|_{s = 0} \\
  = \pdv*{}{s,\bar{s}}{\Big(\log{(\int_{X_s} \prod^{n-2}_{l=0} |z- w_s(z_l)|^{-2\alpha_l} \dzs )}\Big)}\Big|_{s = 0},
\end{multline}
where $X_s = \bC - \aset{0,1,w_s(z_2),\cdots, w_s(z_{n-2})}$.

From now on, we will assume that $w_s(z)$ is smooth with respect to the $s$ and $z$ variables (equivalently $\mu$ is smooth), which is possible as we have recalled in \S~\ref{sub:prep}.
Note that $w_s$ yields a fiberwise diffeomorphism between $\rP_{\Delta_{\epsilon}}$ and $\Delta_{\epsilon} \times \rP_o$.
As a result, we can pull $\eta_s$ back to $\rP_o $, i.e.\ back to $X$, to obtain a pull-back form $w_s^*\eta_s$.
Let $[w_s^*\eta_s]$ be the de Rham class of $w_s^*\eta_s$.
As we have recalled in \S~\ref{sub:prep}, when fixing $z$, $w_s(z)$ is a holomorphic function of $s$.
Therefore, by our construction of $\eta(\bfu)$, (\ref{eq:balpha}), we see that $s \mapsto [w_s^*\eta_s]$ is a holomorphic map from $\Delta_\epsilon$ to $\rH^1(X,F)$.

We are going to expand the r.h.s.\ of (\ref{eq:l-key-eq}).
To bypass the difficulties caused by the singularities of $w_s^*\eta_s$, as in~\cite[Proposition~2.19]{Deligne_1986mo}, we can find smooth section $\sigma_s$ of $F$ supported near the singularities of $w_s^*\eta_s$ such that $\tau_s \coloneqq w_s^*\eta_s - \diff \sigma_s$ has a compact support\footnote{For the small $s$, the supports of $\aset{\tau_s}$ are uniform.} and
\begin{equation}
  \label{eq:wt}
  \opj^{-1} [w_s^*\eta_s] = [\tau_s]\in \rH^1_c(X,F),\quad \int_{X} w_s^*\eta_s \wedge \overline{w_s^*\eta}_s = \int_{X} \tau_s \wedge \overline{\tau}_s.
\end{equation}
As a result, $s \mapsto [\tau_s]$ is also a holomorphic map from $\Delta_\epsilon$ to $\rH^1_c(X,F)$.
Then, similar to~\cite[(7.5)]{Tian_1987aa}, the r.h.s.\ of (\ref{eq:l-key-eq}) can be expanded as follows,
\begin{equation}
  \label{eq:wvv2}
  \begin{aligned}
  -i\omega_{\tv}(v,\bar{v}) =& \pdv*{}{s,\bar{s}}{\Big(\log{\int_{X_s} \frac{\eta_s \wedge \overline{\eta}_s}{-\tpi}}\Big)}\Big|_{s = 0}\\
    =& \pdv*{}{s,\bar{s}}{\Big(\log{\int_{X} \frac{w_s^*\eta_s \wedge \overline{w_s^*\eta}_s}{-\tpi}}\Big)}\Big|_{s = 0}
       = \pdv*{}{s,\bar{s}}{\Big(\log{\int_{X} \frac{\tau_s \wedge \overline{\tau}_s}{-\tpi}}\Big)}\Big|_{s = 0}\\
    =& \frac{\int_X \tau_0 \wedge \overline{\tau}_0 \cdot \int_X \pdv{\tau_s}{s} \wedge \overline{\pdv{\tau_s}{s}}  - \int_X \tau_0 \wedge \overline{\pdv{\tau_s}{s}} \cdot \int_X \pdv{\tau_s}{s} \wedge \overline{\tau}_0}{(\int_X \tau_0 \wedge \overline{\tau}_0)^2}\Big|_{s = 0}.
  \end{aligned}
\end{equation}
In this expression, we also note that
\begin{equation}
  \label{eq:wet}
  \Big[\pdv{\tau_s}{s}\Big|_{s=0}\Big] = \Big[\Big(\pdv{(w^*_s\eta_s)}{s} - \diff (\pdv{\sigma_s}{s})\Big)\Big|_{s=0}\Big] = \opj^{-1}\Big[\pdv{(w^*_s\eta_s)}{s}\Big|_{s=0}\Big].
\end{equation}

In view of (\ref{eq:wet}), to further simplify (\ref{eq:wvv2}), we need to find a good representative element for the class $[\pdv{(w^*_s\eta_s)}{s}|_{s=0}]$.
Keeping this in mind, we introduce $V$ to be the following $(1,0)$-type complex vector field over $\bC$,
\begin{equation}
  \label{eq:fz}
  V(z) \coloneqq \Big(\frac{1}{\tpi} \int_{\bC} \big(a(\zeta) - \mu(\zeta)\big) R(\zeta,z) \diff \zeta \wedge \diff \bar{\zeta}\Big) \frac{\partial}{\partial z}.
\end{equation}
Owing to Lemma~\ref{lm:int2} and Proposition~\ref{prop:two-cl}, we know that $V$ is continuous on $\bC$ and smooth away from $z_k$, $k = 0,\cdots,n-2$.

Meanwhile, by Lemma~\ref{lm:quad}, one can check that $R(\zeta,z_k) \in Q(\Xst)$, $k = 0,\cdots,n-2$.
As a result, the equality (\ref{eq:two-v}) implies that $V$ vanishes at $z_k$.
Moreover, as we have said in Remark~\ref{rk:dy}, we can extend $V$ continuously to a $(1,0)$-type complex vector field on $\hbC$, denoting the extended vector field by $V$ still, and $V(z_{n-1}) = V(\infty) = 0$.

Now, since $\rT_{\bR} \hbC$ is isomorphic to $\rT \hbC$ as real vector bundles, $V$ corresponds to a real continuous vector field on $\hbC$, denoted by $V_{\bR}$.
Fix a small number $\varepsilon_0$.
Let $D_{\varepsilon_0}$ (resp.\ $D_{\varepsilon_0/2}$) be the union of discs centered at $z_k$, $k = 0.\cdots,n-1$ of radius $\varepsilon_0$ (resp.\ $\varepsilon_0/2$) with respect to a fixed metric on $\hbC$.
For any small enough real number $\delta$, we can find a flow $\phi_t$, $|t| < \delta$, defined on $\hbC - D_{\varepsilon_0}$, which is generated by $V_{\bR}$, i.e.\
\begin{equation}
  \label{eq:phit}
  \frac{\diff \phi_t(z)}{\diff t}\Big|_{t=0} = 2V_{\bR}(z),\; \phi_0(z) = z,\;
  \phi_t(z) \subseteq \hbC - D_{\varepsilon_0/2},
\end{equation}
where $z\in \hbC - D_{\varepsilon_0}$.

To be more precise, on $\hbC - D_{\varepsilon_0/2} \subseteq \bC \simeq \bR^2$, we introduce the real coordinates $(x,y)$, which corresponds to the complex coordinate $x + iy$ in the usual way.
With such real coordinates, we can treat $V_{\bR}$ as a function valued in $\bR^2$, i.e.\ $V_{\bR} = (V^1, V^2)$.
Then, $V$, viewed as a complex function, equals to $V = V^1 + iV^2$.
One can check that on $\bC \subseteq \hbC$, such an identification is compatible with the usual isomorphism between $\rT_{\bR} \hbC$ and $\rT \hbC$.
Let $(\phi^1_t, \phi^2_t)$ be the components of $\phi_t$ in the real coordinates.

By writing the complex parameter $s$ as $s = t + i \tau$, we define $\Phi_s \coloneqq \phi^1_t + i\phi^2_t$ for $|s|$ small enough.
Then, by (\ref{eq:phit}) and the relation between $V$ and $V_{\bR}$, $\Phi_s$ satisfies the following property
\begin{equation}
  \label{eq:phis}
  \pdv{\Phi_s(z)}{s}\Big|_{s = 0} = V(z),\; \Phi_0(z) = z,\; \Phi_s(z) \subseteq \hbC - D_{\varepsilon_0/2},
\end{equation}
where $z\in \hbC - D_{\varepsilon_0}$.

\begin{remark}
  \label{rk:iu}
  Since $V$ or $V_{\bR}$ is a continuous vector field defined on $\hbC$, it seems natural to presume that the flow $\Phi_s$ or $\phi_t$ can be extended to the whole $\hbC$.
  However, due to Lemma~\ref{lm:int2}, we know that $V$ is not Lipschitz continuous, but at most H\"older continuous around the point $z_k$, $k = 0,\cdots,n-1$, which leads to a well-known difficulty in the ODE theory, i.e.\ no uniqueness of solutions.
Nevertheless, certain ``infinitesimal uniqueness'' still holds, which is enough for our application.
  Namely, for any $z\in X = \hbC - \aset{z_0,\cdots,z_{n-1}}$, although the definition of $\Phi_s(z)$ may depend on $\varepsilon_0$, $\frac{\partial \Phi_s(z)}{\partial s}|_{s=0}$ does not depend on $\varepsilon_0$.
\end{remark}

Let $z \in X$ and choose $\varepsilon_0$ small enough such that $z \in \hbC - D_{\varepsilon_0}$.
Using $\Phi_s$ and $w_s$, we take their composition $A_s \coloneqq w_s \circ \Phi_s$, which is a diffeomorphism from $\hbC - D_{\varepsilon_0}$ to its image for a fixed $s$.
By (\ref{eq:bel-diff}), (\ref{eq:fz}) and (\ref{eq:phis}), we have
\begin{multline}
  \label{eq:As}
  \pdv{A_s}{s}\Big|_{s=0}(z) = \pdv{w_s}{s}\Big|_{s=0}(\Phi_0(z)) + \pdv{w_0}{z}\pdv{\Phi_s}{s}\Big|_{s=0}(z) + \pdv{w_0}{\zb}\pdv{\overline{\Phi}_s}{s}\Big|_{s=0}(z)\\
  = \Big(\frac{1}{\tpi} \int_{\bC} a(\zeta) R(\zeta,z) \diff \zeta \wedge \diff \bar{\zeta}\Big) \frac{\partial}{\partial z},
\end{multline}
where for the last equality, we also use $w_0(z) = z$.

Now we claim that the derivative of $A_s$ with respect to $s$ at $0$ can give a desired representative of
$[\pdv{(w^*_s\eta_s)}{s}|_{s=0}]$.
\begin{lemma}
  \label{lm:cl-eq}
  \begin{equation}
    \label{eq:cl-eq}
    \Big[\pdv{(w^*_s\eta_s)}{s}\Big|_{s=0}\Big] = \Big[\pdv{(A^*_s\eta_s)}{s}\Big|_{s=0}\Big] \in \rH^1(X, F).
  \end{equation}
\end{lemma}
Here, we would like to emphasize again that as in Remark~\ref{rk:iu}, although the definition of $A_s$ depends on $\varepsilon_0$, the quantities associated with the derivative of $A_s$ at $0$, like $\pdv{A_s}{s}|_{s=0}$ and $\pdv{(A^*_s\eta_s)}{s}|_{s=0}$, are independent of $\varepsilon_0$.

\begin{proof}
  Let $D^*_{\varepsilon_0}\coloneqq D_{\varepsilon_0} - \aset{z_0,\cdots,z_{n-1}}$, a union of small discs with centers removed.
  By the homotopy invariance of the cohomology group, the inclusion $\opi:X - D^*_{\varepsilon_0} \hookrightarrow X$ induces an isomorphism
  \begin{equation*}
    \opi^*:\rH^1(X, F) \simarrow \rH^1(X - D^*_{\varepsilon_0},F).
  \end{equation*}
  Therefore, to show (\ref{eq:cl-eq}), we only need to show that
  \begin{equation}
    \label{eq:cl-eqi}
    \Big[\opi^*\Big(\pdv{(w^*_s\eta_s)}{s}\Big|_{s=0}\Big)\Big] = \Big[\opi^*\Big(\pdv{(A^*_s\eta_s)}{s}\Big|_{s=0}\Big)\Big].
  \end{equation}
  In the rest of the proof, except explicitly stressed, all the forms are defined on $X - D^*_{\varepsilon_0}$ and we will omit $\opi^*$ before the forms.

  By (\ref{eq:phis}), $\Phi_s(X - D^*_{\varepsilon_0}) \subseteq X - D^*_{\varepsilon_0/2}$ for $|s|$ small.
  Hence, $\Phi^*_s(w^*_s\eta_s)$ is a well-defined form on $X - D^*_{\varepsilon_0}$.
  Since $\Phi_s$ is homotopic to the identity map, we have
  \begin{equation*}
    A^*_s\eta_s - w^*_s\eta = \Phi^*_s(w^*_s\eta_s) - w^*_s\eta_s = \diff \gamma_s,
  \end{equation*}
  where $\gamma_s$ is a transgression form, which can be written down explicitly, c.f.~\cite{Bott_1982aa}.
  Therefore,
  \begin{equation*}
    \pdv{(A^*_s\eta_s)}{s}\Big|_{s=0} - \pdv{(w^*_s\eta_s)}{s}\Big|_{s=0} = \diff \Big(\pdv{\gamma_s}{s}\Big|_{s=0}\Big),
  \end{equation*}
  from which (\ref{eq:cl-eqi}) follows.
\end{proof}

Unlike $\pdv{(w^*_s\eta_s)}{s}|_{s=0}$,  the following lemma asserts that $\pdv{(A^*_s\eta_s)}{s}|_{s=0}$ has a clean expression, which itself is a counterpart to (\ref{eq:tian-lm}) or~\cite[Lemma~7.2]{Tian_1987aa}.

\begin{lemma}
  \label{lm:key}
  The derivative of $A^*_s\eta_s$ with respect to $s$ at $0$ has the following expression,
  \begin{equation*}
    \pdv{(A^*_s\eta_s)}{s}\Big|_{s = 0} = c_0 \eta_0 + (a \frac{\partial}{\partial z} \otimes \diff \zb) \mathbin\lrcorner \eta_0,
  \end{equation*}
  where $c_0$ is a constant.\end{lemma}
As in Proposition~\ref{prop:two-cl}, we will denote the quantity $\frac{1}{\tpi} \int_{\bC} a(\zeta) R(\zeta,z) \diff \zeta \wedge \diff \bar{\zeta}$ appearing in (\ref{eq:As}) by $Y$ in the following proof.
\begin{proof}
  As before, we choose $z\in X$ and a small $\varepsilon_0$ to define $A_s$.
  Then, we use $A_s$ to calculate $\pdv{(A^*_s\eta_s)}{s}|_{s = 0}$ in a neighborhood of $z$.
  Besides that, in the following calculation, we will treat $A_s$ and $w_s$ as functions on $X\subseteq \bC$ directly rather than as diffeomorphisms.
  As a result, for derivatives of $A_s$ and $w_s$, unlike the notations used before, we will omit the basis vectors, $\partial/ \partial z$ or $\partial/ \partial \zb$.

  Firstly, we expand $A^*_s\eta_s(z)$ explicitly,
  \begin{equation*}
      A^*_s\eta_s(z) = \prod^{n-2}_{k=0} (A_s(z) - w_s(z_k))^{-\alpha_k}\cdot e \diff A_s(z),
  \end{equation*}
  To calculate the derivatives of the r.h.s.\ of the above equality, we note that due to (\ref{eq:As}),
  \begin{equation*}
    \pdv*{}{s}{\Big(\pdv{A_s}{z}\Big)}\Big|_{s = 0} = \pdv*{}{z}{\Big(\pdv{A_s}{s}\Big|_{s = 0}\Big)} = \pdv{Y}{z},
\end{equation*}
  Similarly, we also have
  \begin{equation*}
    \pdv*{}{s}{\Big(\pdv{A_s}{\zb}\Big)}\Big|_{s = 0} = \pdv{Y}{\zb} = a,
  \end{equation*}
  where we use Lemma~\ref{lm:int2} in the last equality.

  As noted before, by Lemma~\ref{lm:quad}, $R(z,z_k)\in Q(\Xst)$, for $k=0,\cdots,n-2$.
  Hence, due to (\ref{eq:two-v}), we have
  \begin{equation*}
    \pdv{w_s}{s}\Big|_{s =0}(z_k) = \frac{1}{\tpi} \int_{\bC} \mu(\zeta) R(\zeta,z_k) \diff \zeta \wedge \diff \bar{\zeta} = \frac{1}{\tpi} \int_{\bC} a(\zeta) R(\zeta,z_k) \diff \zeta \wedge \diff \bar{\zeta} = Y(z_k).
  \end{equation*}

  Using the above four equalities, near $z$, we have
  \begin{equation}
    \label{eq:exp-phe}
    \begin{aligned}
      \pdv{(A^*_s\eta_s)}{s}\Big|_{s = 0} &=
      \begin{multlined}[t]
        \pdv*{}{s}{\Big(\prod^{n-2}_{k=0} (A_s(z) - w_s(z_k))^{-\alpha_k}\Big)}\Big|_{s = 0}\cdot e \diff z \\
        + \Big(\prod^{n-2}_{k=0} (z - z_k)^{-\alpha_k}\Big) \cdot e \pdv*{}{s}{\Big(\pdv{A_s}{z} \diff z + \pdv{A_s}{\zb} \diff \zb\Big)}\Big|_{s = 0}
      \end{multlined} \\
      &=
      \begin{multlined}[t]
        \Big(- \sum^{n-2}_{k = 0} \frac{Y(z) - Y(z_k)}{z - z_k} \alpha_k + \pdv{Y}{z}(z)\Big) \prod^{n-2}_{k=0} (z - z_k)^{-\alpha_k} \cdot e \diff z \\
        + a(z) \prod^{n-2}_{k=0} (z - z_k)^{-\alpha_k} \cdot e \diff \zb
      \end{multlined}\\
      & = \Big(- \sum^{n-2}_{k = 0} \frac{Y(z) - Y(z_k)}{z - z_k} \alpha_k + \pdv{Y}{z}(z)\Big) \eta_0 + (a \frac{\partial}{\partial z} \otimes \diff \zb) \mathbin\lrcorner \eta_0.
    \end{aligned}
  \end{equation}

  Let
  \begin{equation*}
    D(z) \coloneqq - \sum^{n-2}_{k = 0} \frac{Y(z) - Y(z_k)}{z - z_k} \alpha_k + \pdv{Y}{z}(z).
  \end{equation*}
  Due to (\ref{eq:exp-phe}), to finish the proof of the lemma, we only need to check that $D(z)$ is a constant function.
  Before verifying this point, we first note that $D(z)$, by definition, is a smooth function on $X$, can also be thought as an $\rL^{p-\varepsilon}_{\loc}$ function on $\bC$, where $\varepsilon$ is a small enough number.
  To see this, we notice that Remark~\ref{rk:dy} implies $\pdv{Y}{z}$ is an $\rL^p_{\loc}$ function on $\bC$.
  Moreover, by Lemma~\ref{lm:int2}, $Y(z)$ is a locally $(1-2/p)$-H\"older function on $\bC$, which means that $(Y(z) - Y(z_k))/(z-z_k)$ is also $\rL^{p-\varepsilon}_{\loc}$.

We are going to calculate the distributional $\bar{\partial}$ derivative of $D(z)$.
  As before, since $a \in V(\Xst, g)$, we write $a$ as $\bar{\psi} e^{-\gamma} = \bar{\psi} \prod^{n-2}_{k = 0} |z - z_k|^{2\alpha_k}$, $\psi \in Q(\Xst)$.
  Using this expression, by the direct calculation, we know that $a$ has the $\rL^1_{\loc}$ distributional $\partial$ derivative and
  \begin{equation}
    \label{eq:pdz}
    \pdv{a}{z} = \sum^{n-2}_{k = 0} \big(\frac{\alpha_k}{z - z_k} a\big).
  \end{equation}
  Therefore, since the distributional derivatives commutes, the distributional $\bar{\partial}$ derivative of $\pdv{Y}{z}$ is
  \begin{equation}
    \label{eq:ydzz}
    \pdv{Y}{\zb,z} = \pdv{Y}{z,\zb} = \pdv{a}{z}.
  \end{equation}
  which also implies that $\pdv{Y}{\zb,z}$ is $\rL^1_{\loc}$.

  Next, we deal with the summation terms in $D(z)$.
  Due to the regularity result about $Y(z)$ given in Proposition~\ref{prop:two-cl}, we can check that
  \begin{equation}
    \label{eq:ydz}
    \pdv*{}{\zb}{\Big(\frac{Y(z) - Y(z_k)}{z - z_k}\Big)} = \frac{1}{z - z_k} \pdv{Y}{\zb} = \frac{a}{z - z_k}
  \end{equation}
  as distributions over $\bC$.
  Note that the r.h.s.\ of the above equality is also an $\rL^1_{\loc}$ function.

  Now, by (\ref{eq:pdz}), (\ref{eq:ydzz}) and (\ref{eq:ydz}), the distributional $\bar{\partial}$ derivative of $D(z)$ is
  \begin{equation*}
    \pdv{D}{\zb} = \sum^{n-2}_{k = 0} \big(\frac{-\alpha_k}{z - z_k} a\big) + \pdv{Y}{z,\zb} = \sum^{n-2}_{k = 0} \big(\frac{-\alpha_k}{z - z_k} a\big) + \pdv{a}{z} = 0.
  \end{equation*}
Therefore, Weyl's lemma implies that $D(z)$ is a smooth, consequently holomorphic, function on $\bC$.

  Finally, by Lemma~\ref{lm:int2}, as well as Lemma~\ref{lm:py-asymp} which we will prove in the next subsection, $D(z)$ is asymptotic to $\cO(|z|^{2/p})$ as $z \rightarrow \infty$.
  By Liouville's theorem, we conclude that $D(z)$ must be a constant.
\end{proof}

\subsection{Proof of Theorem~\ref{thm:main}.}
\label{sub:proof-thm-main}
Now, we use Lemma~\ref{lm:key} to further simplify (\ref{eq:wvv2}) and finish the proof of Theorem~\ref{thm:main}.

The main point is to replace the terms in (\ref{eq:wvv2}) involving the derivatives of $w_s$ by that of $A_s$.
As an example, we show the following equality and the other terms can be handled using a similar argument,
\begin{equation}
  \label{eq:w-eq-a}
  \Big(\int_X \pdv{\tau_s}{s} \wedge \overline{\pdv{\tau_s}{s}}\Big)\Big|_{s=0} = \Big(\int_X \pdv{(A^*_s\eta_s)}{s} \wedge \overline{\pdv{(A^*_s\eta_s)}{s}}\Big)\Big|_{s = 0}
\end{equation}
To begin with, using the equality (\ref{eq:wet}), we have
\begin{equation}
  \label{eq:w-eq-t}
  \Big(\int_X \pdv{\tau_s}{s} \wedge \overline{\pdv{\tau_s}{s}}\Big)\Big|_{s=0} = \Big\langle \opj^{-1}\Big(\Big[\pdv{(w^*_s\eta_s)}{s}\Big|_{s=0}\Big]\Big) \cup \opj^{-1}\Big(\Big[\overline{\pdv{(w^*_s\eta_s)}{s}}\Big|_{s=0}\Big]\Big), [X] \Big\rangle.
\end{equation}
Due to (\ref{eq:cl-eq}), to prove (\ref{eq:w-eq-a}), we only need to expand the r.h.s.\ of (\ref{eq:w-eq-a}) in the way like (\ref{eq:w-eq-t}).

To simplify notations, we will denote $\pdv{(A^*_s\eta_s)}{s}|_{s = 0}$ by $\Pi$ in the following.
We rewrite (\ref{eq:exp-phe}) in the following way,
\begin{equation*}
  \Pi = \diff \Big( Y(z) \prod^{n-2}_{k=0} (z - z_k)^{-\alpha_k} \cdot e\Big) + \sum^{n-2}_{k=0}\frac{Y(z_k)\alpha_k}{z - z_k} \eta_0.
\end{equation*}
Fix $k$ and choose $D_{k,\varepsilon}$ to be a small disk centered at $z_k$.
Let $\varphi_k$ to be a cut-off function on $\bC$ equal to $1$ on $D_{k,\varepsilon}$ and vanishing outside $D_{k,2\varepsilon}$.
By the above equality, we can find a holomorphic section $u_k$ of $F$ on $D_{k,\varepsilon} - \aset{z_k}$ such that
\begin{equation*}
  \Pi = \diff\Big( \big(Y(z) - Y(z_k)\big)\prod^{n-2}_{l=0} (z - z_l)^{-\alpha_l} \cdot e \Big) + \diff u_k,\; v_{z_k}(u_k) = -\alpha_k + 1.
\end{equation*}
holds on $D_{k,\varepsilon} - \aset{z_k}$.
Recall that $v_{z_k}(\bullet)$ is the valuation at $z_k$ defined in (\ref{eq:def-val}).
Let
\begin{equation*}
  y_k \coloneqq \big(Y(z) - Y(z_k)\big)\prod^{n-2}_{l=0} (z - z_l)^{-\alpha_l} \cdot e.
\end{equation*}
Then, $\Pi - \diff(\varphi_k(y_k  + u_k))$ vanishes near $z_k$.
Therefore, using Lemma~\ref{lm:key}, we have
\begin{multline}
  \label{eq:int-xd}
  \int_X \Big(\Pi - \diff(\varphi_k(y_k  + u_k))\Big) \wedge \overline{\Pi}
  = \int_{X - D_{k,\varepsilon}} \Pi \wedge \overline{\Pi}
  + \int_{\partial D_{k,\varepsilon}} (y_k + u_k) \wedge (\overline{c_0 \eta_0}) \\
  + \int_{\partial D_{k,\varepsilon}} (y_k + u_k) \wedge \Big(\overline{(a \frac{\partial}{\partial z} \otimes \diff \zb) \mathbin\lrcorner \eta_0}\Big)
\end{multline}

By our definition of $y_k$ and $u_k$, and using Proposition~\ref{prop:two-cl}, we have the estimate
\begin{equation}
  \label{eq:est-yu}
  y_k + u_k =
  \begin{cases}
    \cO(|z - z_k|^{\alpha_k - \delta}),& 0< \alpha_k < 1/2;\\
    \cO(|z - z_k|^{1 - \alpha_k - \delta}),& 1/2\le \alpha_k < 1.
  \end{cases}
\end{equation}
near $z_k$, where $\delta > 0$ is any sufficiently small number.
On the other hand, by direct calculation, near $z_k$, we also have
\begin{equation}
  \label{eq:o-eta}
  \begin{aligned}
    \eta_0 &= \cO(|z- z_k|^{-\alpha_k}),\\ a \frac{\partial}{\partial z} \otimes \diff \zb \mathbin\lrcorner \eta_0 &= \cO(|z- z_k|^{-1 + \alpha_k}).\end{aligned}
\end{equation}
Due to (\ref{eq:est-yu}) and (\ref{eq:o-eta}), we can see that by choosing a very small $\delta$, two boundary integrals in (\ref{eq:int-xd}) converge to $0$ as $\varepsilon \rightarrow 0$.

Now, with $\varepsilon$ in (\ref{eq:int-xd}) going to $0$, we have shown
\begin{equation}
  \label{eq:int-xd2}
  \int_X \big(\Pi - \diff(\varphi_k(y_k  + u_k))\big) \wedge \overline{\Pi}
  = \int_{X} \Pi \wedge \overline{\Pi}
\end{equation}
Note the argument we use to show the above equality is purely local.
We can construct $\varphi_k,y_k,u_k$ near other $k$ using the same method.\footnote{Certainly, near $z_{n-1} = \infty$, we should change the coordinate: $z \mapsto 1/z$.}
Therefore, (\ref{eq:int-xd2}) is also true if we replace $\varphi_k(y_k + u_k)$ by $\sum^{n-1}_{k=0} \varphi_k(y_k + u_k)$.
But, (\ref{eq:int-xd2}) in this form implies the following equality,
\begin{equation*}
  \int_X \Pi \wedge \overline{\Pi} = \langle \opj^{-1}([\Pi]) \cup \opj^{-1}([\overline{\Pi}]), [X] \rangle.
\end{equation*}
Now, the equality (\ref{eq:w-eq-a}) is a consequence of (\ref{eq:cl-eq}), (\ref{eq:w-eq-t}) and the above equality.

Finally, we can replace the derivatives of $w_s$ by that of $A_s$ in (\ref{eq:wvv2}) and use (\ref{eq:wt}) to obtain
\begin{equation*}
  -i\omega_{\tv}(v,\bar{v}) = \frac{\int_X \eta_0 \wedge \overline{\eta}_0 \cdot \int_X \Pi \wedge \overline{\Pi}  - \int_X \eta_0 \wedge \overline{\Pi} \cdot \int_X \Pi \wedge \overline{\eta}_0}{(\int_X \eta_0 \wedge \overline{\eta}_0)^2}.
\end{equation*}
Using Lemma~\ref{lm:key} to expand the r.h.s.\ of the above formula, we have
\begin{equation*}
  \omega_{\tv}(v,\bar{v}) = i\frac{\int_X ({(a \frac{\partial}{\partial z} \otimes \diff \zb) \mathbin\lrcorner \eta_0}) \wedge (\overline{(a \frac{\partial}{\partial z} \otimes \diff \zb) \mathbin\lrcorner \eta_0})}{\int_X \eta_0 \wedge \overline{\eta}_0} = \frac{1}{2\pi} \frac{\int_X |a|^2 e^\gamma \diff z \wedge \diff \zb}{\Vol{(X,g)}} = \omega_{\pw}(v,\bar{v}).
\end{equation*}
Hence, we have shown (\ref{eq:key-eq}) and the proof of Theorem~\ref{thm:main} completes.

At last, we prove an asymptotics result used in the proof of Lemma~\ref{lm:key}.
\begin{lemma}
  \label{lm:py-asymp}
$\pdv{Y}{z}(z) = \cO(|z|^{2/p})$ as $z \rightarrow \infty$.
\end{lemma}
\begin{proof}
  As in the proof of Lemma~\ref{lm:int2}, we will choose $a_1 = a \chi_{\Delta}$, $a_2 = a (1-\chi_{\Delta})$, and write $Y = Y_1 + Y_2$, where
  \begin{equation*}
    Y_1(z) \coloneqq \frac{1}{\tpi} \int_{\Delta} a_1(\zeta) R(\zeta,z) \diff \zeta \wedge \diff \bar{\zeta},\; Y_2(z) \coloneqq \frac{1}{\tpi} \int_{\bC- \Delta} a_2(\zeta) R(\zeta,z) \diff \zeta \wedge \diff \bar{\zeta}.
  \end{equation*}
  As before, by changing variables, we have
  \begin{equation*}
    Y_2(z) = \frac{-z^2}{\tpi} \int_{\Delta} a_2(1/\xi) \frac{\xi^2}{\bar{\xi}^2} R(\xi,1/z)  \diff \xi \wedge \diff \bar{\xi}.
  \end{equation*}
  We deal with the asymptotic behavior of $\pdv{Y_1}{z}$ and $\pdv{Y_2}{z}$ respectively.

  For $|z|$ sufficiently large, by~\cite[Lemma~4.21]{Imayoshi_1992in} or by the direct calculation, the following equality holds,
  \begin{equation*}
    \pdv*{}{z}{\Big(\frac{1}{\tpi} \int_{\Delta} a_1(\zeta) (\frac{1}{\zeta -z} - \frac{1}{\zeta}) \diff \zeta \wedge \diff \bar{\zeta}\Big)} = \frac{1}{\tpi} \int_{\Delta} a_1(\zeta) \frac{1}{(\zeta - z)^2} \diff \zeta \wedge \diff \bar{\zeta}.
  \end{equation*}
  Using the definition of $R$, the above equality implies that
  \begin{equation*}
    |\pdv{Y_1}{z}(z)| \le C_1 + \frac{1}{2\pi} \int_{\Delta} |a_1(\zeta)| \frac{1}{(|z| - 1)^2} |\diff \zeta|^2 \le C_1 + \frac{C_2}{|z|^2}.
  \end{equation*}

  As for $\pdv{Y_2}{z}$, using~\cite[Lemma~4.21]{Imayoshi_1992in} again, we have the following expression for it,
  \begin{multline}
    \label{eq:pd-y2}
    \pdv{Y_2}{z} = \frac{-z}{\pi i} \int_{\Delta} a_2(1/\xi) \frac{\xi^2}{\bar{\xi}^2} (\frac{1}{\xi -z^{-1}} - \frac{1}{\xi}) \diff \xi \wedge \diff \bar{\xi} \\
    + \frac{1}{\tpi} \int_{\Delta} a_2(1/\xi) \frac{\xi^2}{\bar{\xi}^2} \frac{1}{(\xi - z^{-1})^2}  \diff \xi \wedge \diff \bar{\xi} \\
    + \frac{1}{\tpi} \int_{\Delta} a_2(1/\xi) \frac{\xi^2}{\bar{\xi}^2} (\frac{1}{\xi - 1} - \frac{1}{\xi}) \diff \xi \wedge \diff \bar{\xi},
  \end{multline}
  among which the \nth{2} term on the r.h.s.\ is the Hilbert transformation of $a_2(1/\xi) {\xi^2}/{\bar{\xi}^2}$.

  Since $a_2(1/\xi) {\xi^2}/{\bar{\xi}^2} \in \rL^p(\Delta)$, $p > 2$, Lemma~\ref{lm:int} implies the \nth{1} term on the r.h.s.\ of (\ref{eq:pd-y2}) can be estimated as follows,
  \begin{equation}
    \label{eq:pd-y2-1}
    \Big|\frac{-z}{\pi i} \int_{\Delta} a_2(1/\xi) \frac{\xi^2}{\bar{\xi}^2} (\frac{1}{\xi -z^{-1}} - \frac{1}{\xi}) \diff \xi \wedge \diff \bar{\xi} \Big| \le |z| \cdot C_3(|z^{-1} - 0|^{1 - 2/p}) = C_3 |z|^{2/p}.
  \end{equation}

  To deal with the \nth{2} term on the r.h.s.\ of (\ref{eq:pd-y2}), as before, we write $a$ as $\bar{\psi} e^{-\gamma} = \bar{\psi} \prod^{n-2}_{k = 0} |z - z_k|^{2\alpha_k}$ for some $\psi \in Q(\Xst)$.
  Then, using Lemma~\ref{lm:quad}, one can check that $a_2(1/z)z^2/\zb^2$ satisfies the asymptotics assumptions in Lemma~\ref{lm:asym} with $\beta = 1 - 2\alpha_{n-1}$.
  Therefore, by (\ref{eq:p-cond}), the \nth{2} term on the r.h.s.\ of (\ref{eq:pd-y2}) has the estimate $\cO(|z|^{2/p})$ as $z \rightarrow \infty$.
  Combining this with (\ref{eq:p-cond}) and (\ref{eq:pd-y2-1}), we have proved Lemma~\ref{lm:py-asymp}.
\end{proof}

\bibliographystyle{amsplain}

\end{document}